\documentclass[final]{siamltex}
\usepackage[centertags]{amsmath} 
\usepackage{color} 
\usepackage{amssymb} 
\usepackage{epsfig} 
\usepackage{rotating}
\usepackage{multirow}
\usepackage{psfrag}

\usepackage{epstopdf}

\newtheorem{example}{{\bf Example}} 
 
\newtheorem{conditions}[theorem]{{\bf Conditions}}

\usepackage{xcolor}         
\usepackage{etex}
\usepackage{lscape}	
\usepackage{colortbl}
\usepackage{color}
\usepackage{enumerate}
\usepackage[normalem]{ulem}

\definecolor{lightGray}{RGB}{220,220,220}
\definecolor{darkGray}{RGB}{180,180,180}
 
\begin{document} 
 
\newcommand{\myboldmath}[1]{\mathbf #1} 
\newcommand{\dom}[1]{\mbox{dom}\left(#1\right)} 
\newcommand{\range}[1]{\mbox{range}\left(#1\right)} 
\newcommand{\spanne}[1]{\mbox{span}\left(#1\right)} 
\newcommand{\R}{{\mathbb R}} 
\newcommand{\N}{{\mathbb N}} 
\renewcommand{\grad}{\nabla} 
\newcommand{\lb}[1]{\underline{#1}} 
\newcommand{\ub}[1]{\overline{#1}} 
\newcommand{\hessian}[1]{{\grad^2{#1}}} \newcommand{\compl}[1]{\overline{ #1 }}  
  \newcommand{\NonZeroesVec}{\compl{\IndepOfVars}} 
  \newcommand{\IndepOfVars}{{\cal I}} 
  \newcommand{\NonLinVars}{{\cal N}}   \newcommand{\LinVars}{{\cal L}} \newcommand{\indexSet}{{\cal J}}   \newcommand{\BarNonLinVars}{\bar{\cal N}} 
\newcommand{\Setn}{\left\{1, \dots, n\right\}} 
  \newcommand{\setTo}{\leftarrow} 
\newcommand{\func}{\varphi} 
\newcommand{\Order}{{\cal O}} 
\newcommand{\lmax}{\hat{l}} 
\newcommand{\doubleprime}{{\prime\prime}} 
\newcommand{\funcAtPoint}{f(x_0)} 
\newcommand{\gradAtPoint}[1]{{\left.\grad{ #1 }(x)\right|_{x_0}}} 
\newcommand{\hessianAtPoint}[1]{{\left.\hessian{ #1 }(x)\right|_{x_0}}} 
\newcommand{\matrixSet}[1]{{\cal #1 }} 

\newcommand{\IR}{IR} 
\newcommand{\IntervalHessian}{{I\hspace{-1.0mm}{\cal H}}}
\newcommand{\HessianSet}{{\cal H}}
\newcommand{\tight}[1]{{#1}^\star} 
\newcommand{\lambdaaaT}{\Lambda_s} 
\newcommand{\lambdaabba}{\Lambda_t} 
\newcommand{\reduced}[1]{\phi_{ #1 }} 
\newcommand{\reducedHessian}[2]{{\nabla}^{2}_{ #1 } #2} 
\newcommand{\reducedGrad}[2]{{\nabla}_{ #1 } #2} 
\newcommand{\minusset}{\backslash}  
\newcommand{\IgnoredIndices}{{\cal J}} 
\newcommand{\Mn}{\left\{1, \dots, n\right\}}
\newcommand{\yPrimeReduced}[1]{y^{\prime\,( #1 )}}
\newcommand{\assign}{\leftarrow} 

\newcommand{\hyperRec}{\mathcal{B}}
\newcommand{\openDomain}{\mathcal{U}}
\newcommand{\setN}{\mathcal{N}}

\newcommand{\ANew}{{\textrm{A}\!^\dagger}}
\newcommand{\AOld}{\textrm{A}}
\newcommand{\Ge}{\textrm{G}}
\newcommand{\HR}{\textrm{H}}

\newcommand{\NA}{N_{\mathrm{A}}(\func)}
\newcommand{\NGersh}{N_{\mathrm{G}}(\func)}
\newcommand{\NGershB}{N_{\mathrm{bG}}(\func)}
\newcommand{\NH}{N_{\mathrm{H}}(\func)}
\newcommand{\DeltaNA}{N_{\mathrm{\Delta A}}(\func)}

\newcommand{\hessianBackLineK}{[\grad^2{\func}]_{\text{b},k}}
\newcommand{\variableBack}{u}
\newcommand{\setOfUdatesBackLineK}{\U_k}
\newcommand{\setOfUdatesBackLineKDimD}{\setOfUdatesBackLineK^d}

\newcommand{\RuleSparseEvIntervaladdI}{ 1 }
\newcommand{\RuleSparseEvIntervalexpII}{ 2 }
\newcommand{\RuleSparseEvIntervalsquareII}{ 3 }
\newcommand{\RuleSparseEvIntervaloneOverI}{ 4 }

\newcommand{\dev}{\mathrm{dev}} 

\newcommand{\comment}[1]{ 
  \marginpar{ 
    \fbox{ 
      \parbox{2cm}{ #1 } 
    } 
  } 
} 
 
\newcommand{\WarningBox}[1]{ 
  \fbox{ 
    \parbox{\textwidth}{ 
      #1 
    } 
  } 
} 
\renewcommand{\WarningBox}[1]{} 
 
\definecolor{Gray}{gray}{0.9} 
\newcommand{\AdditionalInfoBox}[1]{ 
  \colorbox{Gray}{ 
    \parbox{\textwidth}{ 
      #1 
    } 
  } 
} 
\newcommand{\CommentNewOld}[2]{\textcolor{red}{#1 \sout{#2}}}
\newcommand{\CommentMmoBeforeSubmission}[2]{\textcolor{red}{#1}} \newcommand{\MmoSecondRound}[1]{
  \fbox{ 
    \parbox{\textwidth}{
      \textcolor{red}{
        #1
      }
    } 
  } 
} 
\renewcommand{\MmoSecondRound}[1]{}      
\newcommand{\MmoThirdRound}[1]{
  \fbox{ 
    \parbox{\textwidth}{
      \textcolor{red}{
        #1
      }
    } 
  } 
} 
\renewcommand{\MmoThirdRound}[1]{}
\title{Improved automatic computation\\ of Hessian matrix spectral bounds}
\author{Moritz Schulze Darup and Martin M\"onnigmann} 
\maketitle 
\begin{abstract} 
This paper presents a fast and powerful method for the computation of eigenvalue bounds for Hessian matrices $\hessian{\func(x)}$ of nonlinear functions $\func:\openDomain\subseteq\R^n\rightarrow\R$
on hyperrectangles $\hyperRec \subset \openDomain$. The method is based on a recently proposed procedure \cite{Monnigmann2011a} for an efficient computation of spectral bounds using extended codelists.
Both the approach from \cite{Monnigmann2011a} and the one presented here substantially differ from established methods 
in that they do deliberately not use any interval matrices and thus result in a favorable numerical complexity of order $\Order(n)\,N(\func)$, where $N(\func)$ denotes the number of operations needed to evaluate $\func$ at a point in its domain.

We improve the method presented in \cite{Monnigmann2011a} by exploiting sparsity, which naturally arises in the underlying codelists.  
The new method provides bounds that are as good as, or better than those from the most accurate existing method in about $82\,\%$ of the test cases.
\end{abstract}

\section{Introduction}
\label{sec:Introduction}

\MmoSecondRound{\begin{itemize}
  \item Check for self plagiarism. \textcolor{blue}{Die Einleitung enthaelt viele kopierte Saetze. Der restliche Text nicht.}
  \item Captions vereinheitlichen. Einige geaendert, aber Aenderungen nicht uebertragen auf die uebrigen.
  \item The 17 cases for the sparse product need to be commented. Why is the list of 17 cases complete? (I am aware that the corresponding question is answered for the sum, but that's not enough, I am afraid.)
  \textcolor{blue}{Die Liste fuer die Multiplikation ist nicht komplett. Es lassen sich theoretisch noch mehr Spezialfaelle auffinden, fuer die wir aber noch keine besseren Einschluesse kennen. Ich glaube es ist utopisch dem Leser alle 17 Faelle zu erlaeutern. Wer die Summe verstanden hat, kann m.E. auch grob die 14 nicht Spezialfaelle identifizieren. In Hinblick auf das Lemma in Anhang lassen sich dann auch die 3 verbleibenden nachvollziehen.}
  \item \textcolor{blue}{Verweis auf reelle Eigenwertschranken-Paper einbauen.}
\end{itemize}
}

We present important improvements for a recently proposed method (see \cite{Monnigmann2011a}) for the efficient calculation of spectral bounds for Hessian matrices on hyperrectangles.
The improvements build on a systematic treatment of sparsity of the involved matrices, which will be shown to result in significantly tighter eigenvalue bounds. The problem can concisely be summarized as follows. 
Let $\func:\openDomain\subseteq\R^n\rightarrow\R$ be a twice continuously differentiable function on an open set $\openDomain\subseteq\R^n$ and let
$\hyperRec= [\lb{x}_1, \ub{x}_1]\times\cdots\times[\lb{x}_n, \ub{x}_n]$ be a closed hyperrectangle in $\openDomain$. 
We seek bounds $\lb{\lambda},\ub{\lambda}\in\R$ such that the relations $\lb{\lambda} \leq \lambda \leq \ub{\lambda}$ hold for all eigenvalues $\lambda$ of all matrices $ H \in \{\hessian{\func(x)}\,| \,x \in \hyperRec \}$. 
More precisely, we solve the following problem: 
\begin{equation}
\label{eq:EigenboundsOnBox}
\text{Find} \,\, \lb{\lambda},\ub{\lambda}\in\R \,\, \text{such that} \,\,
    \lb{\lambda} \leq \min_{x \in \hyperRec} \lambda_{\min}(\hessian{\func}(x)) \,\,\, \text{and}  \,\,\,
    \max_{x \in \hyperRec} \lambda_{\max}(\hessian{\func}(x)) \leq  \ub{\lambda},
\end{equation}
where $\lambda_{\min}(H)$ and $\lambda_{\max}(H)$ denote the smallest and largest eigenvalue, respectively, of the symmetric matrix $H\in \R^{n \times n}$.
A bound $\lb{\lambda}$ (resp.\ $\ub{\lambda}$) is called {\it tight} if there exists at least one $x \in \hyperRec$ such that $\lb{\lambda} = \lambda_{\min}(\hessian{\func}(x))$ (resp. $\ub{\lambda} = \lambda_{\max}(\hessian{\func}(x))$). Note that the problem statement~\eqref{eq:EigenboundsOnBox} does not necessarily imply that $\lb{\lambda}$ and $\ub{\lambda}$ are tight. 

Eigenvalue bounds $\lb{\lambda}$, $\ub{\lambda}$ are used, for example, in numerical optimization methods to detect convexity, or to construct convex underestimators of nonconvex functions~\cite{Adjiman1998b,Adjiman1998a,Androulakis1995}. 
If (\ref{eq:EigenboundsOnBox}) yields $\lb{\lambda}\ge 0$ then $\func$ is convex on the
interior of the hyperrectangle $\hyperRec$. While no conclusion on the convexity can be drawn if 
\eqref{eq:EigenboundsOnBox} results in $\lb{\lambda}<0$, the bound $\lb{\lambda}$ can still be used to construct a convex underestimator for $\func$ on $\hyperRec$. Specifically, 
\begin{equation}\label{eq:alphaBB}
 \breve{\func}(x)= \func(x)- \frac{1}{2}\,\lb{\lambda}\,\sum \nolimits_{i=1}^n \left(\lb{x}_i- x_i\right)\left(\ub{x}_i-x_i\right)
\end{equation}
is convex, coincides with $\func$ at the vertices of $\hyperRec$, and bounds $\func$ from below everywhere else in $\hyperRec$. 
Since a large fraction of the total computation time is spent on the calculation of convex underestimators in global optimization methods~\cite{Adjiman1998b}, fast methods for solving \eqref{eq:EigenboundsOnBox} are of interest.
We briefly note that~\eqref{eq:EigenboundsOnBox} must also be solved in certain problems in automatic control
and systems theory. An illustrative example is given in~\cite{Monnigmann2008b}. 

Existing approaches to solving 
\eqref{eq:EigenboundsOnBox} proceed in two steps: First, 
a symmetric interval matrix (also called \textit{interval Hessian}) that contains all Hessians $\hessian{\func(x)}$ on $\hyperRec$ is calculated: 
\begin{equation}
\label{eq:symIntervalMatrix}  
\begin{array}{c}
\text{Find} \,\, \lb{H}=\lb{H}^T, \ub{H}=\ub{H}^T \! \in \R^{n \times n}  \,\, \text{such that} \,\,
    \lb{H}_{ij} \leq \left(\hessian{\func}(x)\right)_{ij} \leq \ub{H}_{ij} \\
     \text{for every} \,\, i,j \in \{1,\dots, n\} \,\,\text{and every} \,\, x \in \hyperRec.
      \end{array}
\end{equation}  
This task can efficiently be carried out 
by combining interval
arithmetics (IA, see \cite{Neumaier2008}, for example) and automatic differentiation (AD, see \cite{Fischer1995,Rall1981}, for example). 
In the second step, 
spectral bounds can be found by solving the following problem,  
which is similar to, but different from~\eqref{eq:EigenboundsOnBox}: 
  \begin{equation}
\label{eq:EigenboundsSymIntervalMatrix}
\text{Find} \,\, \lb{\lambda},\ub{\lambda}\in\R \,\, \text{such that} \,\,
    \lb{\lambda} \leq \min_{H \in \mathcal{H}} \lambda_{\min}(H) \,\,\, \text{and}  \,\,\,
    \max_{H \in \mathcal{H}} \lambda_{\max}(H) \leq  \ub{\lambda},
\end{equation}
where $\mathcal{H} = \{ H \in \R^{n \times n} \,|\, H_{ij} \in [\lb{H}_{ij},\ub{H}_{ij}], \, H=H^T \}$ is the set of all symmetric matrices that respect the bounds $\lb{H}$ and $\ub{H}$.
Various approaches exist to solving \eqref{eq:EigenboundsSymIntervalMatrix} (see, e.g., \cite{Adjiman1998a,Gershgorin1931,Hertz1992,Rohn1994}). However, since $\{\hessian{\func(x)}\,| \,x \in \hyperRec \} \subseteq \mathcal{H}$, problem \eqref{eq:EigenboundsSymIntervalMatrix} is conservative compared to the original problem~\eqref{eq:EigenboundsOnBox}.
In fact it is the very point of the method introduced in \cite{Monnigmann2011a} and refined in the present paper to avoid computing interval matrices of the form~\eqref{eq:symIntervalMatrix} when solving~\eqref{eq:EigenboundsOnBox} in order to avoid this conservatism.  

We briefly summarize the computational complexity of the existing methods.  
The computation of the matrices $\lb{H}$ and $\ub{H}$ in~\eqref{eq:symIntervalMatrix}  requires $\Order(n^2)\,N(\func)$ (resp.~$\Order(n)\,N(\func)$) operations if the forward (resp.~backward) mode of AD is used, where
$N(\func)$ denotes the number of operations needed to evaluate $\func$ at a point in its domain \cite{Fischer1995}.
After the interval Hessian has been calculated, 
solving~\eqref{eq:EigenboundsSymIntervalMatrix} requires between $\Order(n^2)$ operations for the interval variant of Gershgorin's circle criterion \cite{Adjiman1998a,Gershgorin1931} and $\Order(2^{n}\,n^3)$ operations for Hertz and Rohn's method \cite{Hertz1992,Rohn1994}. 
The latter method is an important benchmark in that it yields \textit{tight} spectral bounds for the matrix set $\mathcal{H}$. 
Albeit the conservatism in~\eqref{eq:symIntervalMatrix}, Hertz and Rohn's method therefore provides the best possible option to solve~\eqref{eq:EigenboundsOnBox} via \eqref{eq:symIntervalMatrix} and \eqref{eq:EigenboundsSymIntervalMatrix}.

The total numerical effort of any approach that uses~\eqref{eq:symIntervalMatrix} and~\eqref{eq:EigenboundsSymIntervalMatrix} corresponds to the sum of the efforts for calculating $\lb{H},\ub{H}$ and solving~\eqref{eq:EigenboundsSymIntervalMatrix}. 
Thus, the numerical effort for the established methods varies between $\Order(n)\,N(\func)+\Order(n^2)$ (backward mode AD combined with Gershgorin's circle criterion) and $\Order(n^2)\,N(\func)+ \Order(2^{n}\,n^3)$ operations (forward mode AD combined with Hertz and Rohn's method). 
The major advantage of the direct method presented in \cite{Monnigmann2011a} is its low computational complexity, 
which was shown to be of order $\Order(n)\,N(\func)$. 

It is the purpose of this paper to improve the method introduced in \cite{Monnigmann2011a} such that sparsity can exploited to find tighter eigenvalue bounds. The improvements do not increase the numerical effort compared to the original method in~\cite{Monnigmann2011a}. In fact, sparsity needs to be investigated once during the automatic generation of the extended codelist. The computations required to evaluate the codelist to obtain eigenvalue bounds on a specific hyperrectangle $\hyperRec$ are no more expensive than those for the 
non-sparse case treated in \cite{Monnigmann2011a}. 
While the computational effort remains the same, the improved method results in significantly tighter eigenvalue bounds than the original procedure from \cite{Monnigmann2011a}. 
To show this, we investigate 1522 examples (taken from the COCONUT collection \cite{Shcherbina2003}, see \cite{SchulzeDarup2014JGO} for details) and compare the eigenvalue bounds resulting from the improved procedure to those obtained with the original one \cite{Monnigmann2011a} and to bounds obtained with the interval Hessian~\eqref{eq:symIntervalMatrix} and Gershgorin's circle criterion and Hertz and Rohn's method. 

We summarize the major aspects of the direct method for the computation of eigenvalue bounds from~\cite{Monnigmann2011a} in Sect.~\ref{sec:oldArithmetic}.
Our main result, the exploitation of sparsity for the improvement of the eigenvalue bounds from~\cite{Monnigmann2011a}, is stated in Sect.~\ref{sec:ImprovedEigenvalueBounds}.  
We analyze 1522 numerical examples from \cite{Shcherbina2003,SchulzeDarup2014JGO} in Sect.~\ref{sec:Examples}\footnote{Results were obtained with \texttt{Jcodegen}, a code generator available from the authors on request, or to be used online on www.rus.rub.de/software/jcodegen. \texttt{Jcodegen} generates ANSI-C code for the algorithm described in Prop.~\ref{prop:newArithmetic} for a given function $\func$. In particular sparsity is treated automatically. The specific hyperrectangle $\hyperRec$ is passed to the resulting code as a runtime parameter.}.
Conclusions are given in Sect.~\ref{sec:Conclusion}.

\section{Notation and Preliminaries}
\label{sec:NotationPreliminaries}

We frequently use index sets $\indexSet \subseteq \setN$, where $\setN:=\N_{1,n}$ and where $\N_{m,n} :=\{i \in \N \,|\, m \leq i \leq n\}$.
The complement of an index set $\indexSet$ is defined as $\indexSet^c:=\setN \setminus \indexSet$.
The cardinality of an index set $\indexSet$ is denoted by $|\indexSet|$.

It is convenient to state eigenvalue bounds as intervals (e.g. $\lambda\in[\lb{\lambda}, \ub{\lambda}]$). 
Intervals $[\lb{a}, \ub{a}] \subset \R$ with $\lb{a}\leq \ub{a}$ are further abbreviated by 
$[a]:=[\lb{a}, \ub{a}]$ whenever appropriate. Interval equality
$[\lb{a}, \ub{a}]= [\lb{b}, \ub{b}]$, is understood to mean $\lb{a}= \lb{b}$ and 
$\ub{a}= \ub{b}$. We frequently carry out calculations on intervals with
standard interval arithmetics (IA) rules. The required rules are summarized in Lem.~\ref{lem:intervalArithmetic} and Tab.~\ref{tab:intervalArithmetic}.

\begin{lemma}[basic interval operations~\cite{Neumaier2008}]
\label{lem:intervalArithmetic}
Let $[a]= [\lb{a}, \ub{a}]$ and $[b]= [\lb{b}, \ub{b}]$ be intervals in $\R$. Let 
  $a\in[\lb{a},\ub{a}]$, $b\in[\lb{b}, \ub{b}]$, and $c\in \R$ be arbitrary real
  numbers. Then, the relations in the second column of Tab.~\ref{tab:intervalArithmetic} hold under the additional restrictions stated in the last column.
\end{lemma}

\begin{table}[htp] 
\caption{Basic interval arithmetic.}
\label{tab:intervalArithmetic}
\centering
\small
\begin{tabular}{c|ll|ll|c}
no. & \multicolumn{2}{c|}{operation / bounds}  & \multicolumn{2}{c|}{definition} & restriction\\
\hline 
\hline
1 & $a+b$ & $\in[a]+[b]$ & \multicolumn{2}{l|}{$:=[\lb{a}+\lb{b},\ub{a}+\ub{b}]$} \\
2 & $a\,b$ & $\in[a]\,[b]$ & \multicolumn{2}{l|}{$:=[\min\{\lb{a}\,\lb{b},\lb{a}\,\ub{b},\ub{a}\,\lb{b},\ub{a}\,\ub{b}\},\max\{\lb{a}\,\lb{b},\lb{a}\,\ub{b},\ub{a}\,\lb{b},\ub{a}\,\ub{b}\}]$} \\
3 & $1/a$ & $\in 1/[a]$ & \multicolumn{2}{l|}{$:=[1/\ub{a},1/\lb{a}]$} & $0 \notin [a]$ \\
4 & $a^m$ & $\in[a]^m$ & $:= [\lb{a}^m,\ub{a}^m]$ & if $\lb{a}>0$ or $m$ odd \\
& & & $:= [\ub{a}^m,\lb{a}^m]$ & if $\ub{a}<0$ and $m$ even \\
& & & $:= [0,\max\{\lb{a}^m,\ub{a}^m \}]$ & otherwise\\
5 & $\sqrt{a}$ & $\in\sqrt{[a]}$ & $:=[\sqrt{\lb{a}},\sqrt{\ub{a}}]$ & & $\lb{a} \geq 0$\\
6 & $\exp{a}$ & $\in\exp([a])$ & $:=[\exp(\lb{a}),\exp(\ub{a})]$ & & \\
7 & $\ln(a)$ & $\in\ln([a])$ & $:=[\ln(\lb{a}),\ln(\ub{a})]$ & & $\lb{a} > 0$\\
8 & $a+c$ & $\in[a]+c$ & $:=[\lb{a}+c,\ub{a}+c]$ & & \\
9 & $c\,a$ & $\in c\,[a]$ & $:=[c\,\lb{a},c\,\ub{a}]$ & if $c \geq 0$ \\
& & & $:=[c\,\ub{a},c\,\lb{a}]$ & otherwise
\end{tabular} 
\end{table}

A lower case letter surrounded by brackets may refer to a real interval $[x]=[\lb{x},\ub{x}] \subset \R$ (as introduced above) or 
a hyperrectangle $[x]=[\lb{x}_1,\ub{x}_1] \times \dots \times [\lb{x}_n,\ub{x}_n] \subset \R^n$ (with $n \geq 2$).
In the latter case, the interval operations listed in Tab.~\ref{tab:intervalArithmetic} are understood to apply to every component. 
For a hyperrectangle $[x]\subset \R^n$ and a nonempty index set $\indexSet \subseteq \setN$ with the $m$ elements $j_1<\dots<j_m$, the term $[x_\indexSet]$ refers to the hyperrectangle
$[x_\indexSet] = [x_{j_1}] \times \dots \times [x_{j_m}]$.

It is furthermore convenient to use \textit{null matrices} in $\R^{m \times r}$, which we denote by $0_{m,r}$, when dealing with sparsity. 
For the special cases $m=0$ or $r=0$ we obtain an \textit{empty matrix}. 
Formally, the empty square matrix $0_{0,0}$ has no eigenvalues. It proves useful to assign the eigenvalue bounds $[\lambda]=[0,0]$ to it. 
Finally, the Cartesian unit vector along the $k$-th direction is denoted by $e_k \in \R^n$. 

\section{Direct computation of eigenvalue bounds for Hessian matrices on hyperrectangles}
\label{sec:oldArithmetic}

We summarize the method introduced in~\cite{Monnigmann2011a} for the direct solution of~\eqref{eq:EigenboundsOnBox} as needed in the present paper. 
We  assume the function $\func$ can be evaluated at an arbitrary point $x\in\openDomain$ by carrying out a finite sequence of operations of the form 
\begin{equation}
\label{eq:codelist}
  \begin{array}{rcl}
        y_1 &=& x_1 \\ 
            &\vdots& \\ 
        y_n &=& x_n \\ 
    y_{n+1} &=& \Phi_{n+1} (y_1, \dots, y_n) \\ 
    y_{n+2} &=& \Phi_{n+2} (y_1, \dots, y_n, y_{n+1}) \\ 
            &\vdots& \\  
    y_{n+t} &=& \Phi_{n+t} (y_1, \dots, y_n, y_{n+1}, \dots, y_{n+t-1}) \\ 
    \func &=& y_{n+t} 
  \end{array}
\end{equation}
where each $\Phi_{n+k}$, $k=1,\dots,t$, represents one of the elementary operations listed in the first column of Tab.~\ref{tab:oldArithmetic}. 
We treat the same operations as in~\cite{Monnigmann2011a} for ease of comparison. Note that additional unary operations can be added according to the rules given in \cite{Monnigmann2011a}. 
We refer to~\eqref{eq:codelist} as the \textit{codelist} of the function $\varphi$.

The codelist~\eqref{eq:codelist} can be used to evaluate the function value $\func(x)$ at a specific point $x$ in its domain.
Using automatic differentiation (AD) \cite{Rall1981} the codelist~\eqref{eq:codelist} can be extended in such a way that the gradient $\grad \func(x)$ or the Hessian $\hessian \func(x)$ at the point $x$ are calculated.
Moreover, using AD and interval arithmetic (IA), \eqref{eq:codelist} can be modified such that interval extensions, interval gradients or interval Hessians of $\func$ on hyperrectangles $\hyperRec\subset \openDomain$
are computed. In fact, extended codelists are commonly used to solve problem~\eqref{eq:symIntervalMatrix} as part of the established procedures for the computation of eigenvalue bounds (see, e.g., \cite{Adjiman1998a}).
In contrast, the method introduced in~\cite{Monnigmann2011a} only requires the interval gradient, but not the interval Hessian.
Essentially, the codelist is extended by arithmetic operations that compute the eigenvalue bounds for the Hessian of the intermediate function in every codelist line.
Formally, this leads to the extended codelist which we introduce in the following theorem. 
\begin{theorem}[algorithm for direct eigenvalue bound computation {\cite[Prop. 4.2]{Monnigmann2011a}}]
\label{thm:directArithmetic}
Assume $\func$ is twice continuously differentiable on $\openDomain$ and can be written as a codelist~\eqref{eq:codelist}.
Let $\hyperRec = [x_1] \times \dots \times [x_n] \subset \openDomain$  be arbitrary.
  Then, for all $x\in \hyperRec$, we have $\func(x)\in [\func]$,
  $\grad{\func}(x)\in[\grad{\func}]$, 
  and $[\lambda_{\min}(\hessian{\func}(x)),\lambda_{\max}(\hessian{\func}(x))] \subseteq [\lambda_\func]$, where 
  $[\func]$, $[\grad{\func}]$, and $[\lambda_\func]$ are calculated 
  by the following algorithm.  
  \begin{enumerate} 
  \item For $k= 1, \dots, n$, 
    set $[y_k]= [\lb{x}_k, \ub{x}_k]$,
    $[\grad y_k]= [e_k,e_k]$, and $[\lambda_k]= [0, 0]$. 
  \item For $k= n+1, \dots, n+t$,
    calculate $[y_k]$, $[\grad y_k]$ and $[\lambda_k]$
    according to the third, fourth, and fifth column of 
    Tab.~\ref{tab:oldArithmetic}.  
  \item Set $[\func]= [y_{n+t}]$, 
    $[\grad{\func}]= [\grad y_{n+t}]$, 
    and $[\lambda_\func]= [\lambda_{n+t}]$. 
  \end{enumerate} 

\end{theorem}

In Tab.~\ref{tab:oldArithmetic}, we use the interval operators $[\lambdaaaT([a])]$ and $[\lambdaabba([a],[b])]$, which are defined according to
\begin{align} 
  \label{eq:lambdaaaT}   
[\Lambda_s([a])]  &= \left\{
\begin{array}{ll}
\,[a]^2  & \text{if} \,\,\, m=1, \\
\,[0, \sum_{i= 1}^m \max \{\lb{a}_i^2, \ub{a}_i^2 \} ] \qquad \quad \, & \text{otherwise},
\end{array}\right.  \\
\label{eq:Lambda_abbaNotation} 
 \textrm{and} \quad [\Lambda_t([a], [b])] &= 
  \left\{  \begin{array}{ll}
\,2\,[a]\,[b]  & \text{if} \,\,\, m=1, \\
\,[-\beta, \beta]+\sum_{i=1}^m [\lb{a}_i, \ub{a}_i]\,[\lb{b}_i, \ub{b}_i]   & \text{otherwise}
\end{array}\right. 
\end{align} 
for hyperrectangles $[a],[b] \subset \R^m$, where $\beta= \sqrt{
    (\sum_{i=1}^m \max \{\lb{a}_i^2, \ub{a}_i^2 \}) 
    (\sum_{i=1}^m \max \{\lb{b}_i^2, \ub{b}_i^2 \}) }$.
We refer to \cite[Lems. 2.2 and 2.3]{Monnigmann2011a} for  details on $[\lambdaaaT([a])]$ and $[\lambdaabba([a],[b])]$.

\begin{table}[htp] 
\caption{Rules for the calculation of $y_k$, $[y_k]$, $[\grad y_k]$ and $[\lambda_k]$ in the $k$-th line of the codelist~\eqref{eq:codelist}. $[\grad y_k]$ refers to the interval gradient of line $k$ with respect to $x$. The interval operators $[\lambdaaaT([a])]$ and $[\lambdaabba([a], [b])]$ are defined in~\eqref{eq:lambdaaaT} and~\eqref{eq:Lambda_abbaNotation}.}
\label{tab:oldArithmetic}
    \centering
\small
\begin{tabular}{l@{\,\,}|l@{\,}|l@{\,}|l@{\,}|l} 
  {\tt op} $\Phi_k$& $y_k$        & $[y_k]$                  
                                  & $[\grad y_k]$ 
                                  & $[\lambda_k]$ \\\hline \hline 
  {\tt var}        & $x_k$        & $[x_k]$                      
                                  & $[e_k,e_k]$
                                  & $[0,0]$ \\
\hline   {\tt add}        & $y_i+ y_j$   & $[y_i]+ [y_j]$         
                                  & $[\grad y_i]+ [\grad y_j]$
                                  & $[\lambda_i]+ [\lambda_j]$  \\ 
  {\tt mul}        & $y_i\,y_j$   & $[y_i]\,[y_j]$           
                                  & $[y_j] [\grad y_i]\!+\! [y_i] [\grad y_j]$
                                  & $[y_j] [\lambda_i]+ [y_i] [\lambda_j]+ [\lambdaabba([\grad y_i], [\grad y_j])]$ \\ 
\hline
  {\tt powNat}     & $y_i^m$      & $[y_i]^m$          
                                  & $m\,[y_i]^{m-1}\,[\grad y_i]$
                                  & $m[y_i]^{m-2}((m\!-\!1)[\lambdaaaT([\grad y_i])]\!+\![y_i] [\lambda_i])$ \\ 

  {\tt oneOver}    & $1/y_i$      & $1/[y_i]$      
                                  & $-[y_k]^2\,[\grad y_i]$ 
                                  &  $[y_k]^2\,(2\,[y_k]\,[\lambdaaaT([\grad y_i])]-[\lambda_i])$ \\ 

  {\tt sqrt}       & $\sqrt{y_i}$ & $[\sqrt{[y_i]}]$           
                                  & $1/(2\,[y_k]) \,[\grad y_i]$
                                  & $1/(2\,[y_k])([\lambda_i]\!+\!1/(-2\,[y_i])[\lambdaaaT([\grad y_i])])$  \\ 
  {\tt exp}        & $\exp(y_i)$  & $[\exp([y_i])]$            
                                  & $[y_k]\,[\grad y_i]$
                                  & $[y_k]\,([\lambdaaaT([\grad y_i])]+[\lambda_i])$ \\ 
  {\tt ln}         & $\ln(y_i)$   & $[\ln([y_i])]$           
                                  & $1/[y_i]\,[\grad y_i]$
                                  & $1/[y_i]\,([\lambda_i]-1/[y_i]\,[\lambdaaaT([\grad y_i])])$  \\
                                    {\tt addC}   & $y_i+ c$     & $[y_i]\!+\![c,c]$      
                                  & $[\grad y_i]$ 
                                  & $[\lambda_i]$ \\
  {\tt mulByC} & $c\,y_i$     & $c\,[y_i]$       
                                  & $c\,[\grad y_i]$
                                  & $c\,[\lambda_i]$ 
\end{tabular} 
\end{table}

\MmoSecondRound{Double check if following lemma is really needed and what for. \textcolor{blue}{Lemma auskommentiert!!! Das Lemma wurde im Beweis von  Prop. 4.11 verwendet. Ich habe es nochmal gesondert aufgefuerht, da die Originale \cite[Lems. 2.2 and 2.3]{Monnigmann2011a} auf die Dimension $n$ baut. Im Beweis in Prop. 4.11 findet sich nun der direkte Verweis. Den Dimensionssprung bekommt der interessierte Leser auch hin.}}    

\section{Improved computation of eigenvalue bounds using sparsity}
\label{sec:ImprovedEigenvalueBounds}

If sparsity is exploited, tighter eigenvalue bounds can be obtained than those
that result from the method summarized in Section~\ref{sec:oldArithmetic}. 
This is evident from the following motivating example.

\begin{example}[method from~\cite{Monnigmann2011a} applied to  $\func(x)=x_1^2 + x_2^2$]
\label{example:BadEigvalApproxForSum}
Consider the function $\func:\R^2 \rightarrow \R$ with $\func(x)=x_1^2 + x_2^2$. 
Theorem~\ref{thm:directArithmetic} results in the following extended codelist. Note that the expressions for $y_k$ listed in~\eqref{eq:extendedCodelistExample1} do not result
from Thm.~\ref{thm:directArithmetic}, but are only given for illustration of the codelist~\eqref{eq:codelist} of $\func$.
\begin{equation}
\label{eq:extendedCodelistExample1}
\small
\begin{tabular}{l|l|l|l|l} 
 $k$ &  $y_k$        &  $[y_k]$   & $[\grad y_k]$  &  $[\lambda_k]$ \\\hline \hline 
1& $x_1 $          & $[x_1]$                      
                                  & $[e_1,e_1]$ 
                                  & $[0,0]$ \\
2 & $x_2 $        & $[x_2]$                       
                                  & $[e_2,e_2]$ 
                                  & $[0,0]$ \\

3 & $y_1^2$     & $[y_1]^2$          
                                  & $2\,[y_1]\,[\grad y_1]$ 
                                  & $2([\lambdaaaT([\grad y_1])]+[y_1]\,[\lambda_1])$ \\ 
                                  
4 & $y_2^2$      & $[y_2]^2$        
                                  & $2\,[y_2]\,[\grad y_2]$ 
                                  & $2([\lambdaaaT([\grad y_2])]+[y_2]\,[\lambda_2])$  \\ 
                                  
5 & $y_3+ y_4$    & $[y_3]+ [y_4]$        
                                  & $[\grad y_3]+ [\grad y_4]$ 
                                  & $[\lambda_3]+ [\lambda_4]$ \\ 
                                  \hline
                                  & $\func = y_5$ & $[\func]=[y_5]$ & $[\grad \func]=[\grad y_5]$ & $[\lambda_\func]=[\lambda_5]$  
\end{tabular} 
\end{equation}
Evaluating the extended codelist~\eqref{eq:extendedCodelistExample1} for the hyperrectangle $\hyperRec = [0,1] \times [0,1]$ by computing 
 $[y_k]$, $[\grad y_k]$, and $[\lambda_k]$ and storing the results line by line yields
\begin{equation}
\label{eq:extendedCodelistExample1Results}
 \small
\begin{tabular}{l@{\qquad}l@{\qquad}l} 
 $[y_1]= [0,1]$,   & $[\grad y_1]= ( [1,1] , [0,0] )^T$, & $[\lambda_1]=[0,0]$, \\
 $[y_2]= [0,1]$,     &$[\grad y_2]= ( [0,0] , [1,1] )^T$,
                                 & $[\lambda_2]=[0,0]$, \\

 $[y_3]= [0,1]$,    & $[\grad y_3]= ( [0,2] , [0,0] )^T$, & $[\lambda_3]=[0,2]$, \\ 
                                  
 $[y_4]= [0,1]$,   & $[\grad y_4]= ( [0,0] , [0,2] )^T$, &  $[\lambda_4]=[0,2]$, \\ 
                                 
$[y_5]= [0,2]$,   & $[\grad y_5]= ( [0,2] , [0,2] )^T$, & $[\lambda_5]=[0,4]$, 
\end{tabular} 
\end{equation}
where $[\lambdaaaT([\grad y_1])]=[0,1]$ and $[\lambdaaaT([\grad y_2])]=[0,1]$ according to Eq.~\eqref{eq:Lambda_abbaNotation}.
Thus, we obtain the eigenvalue bounds $[\lambda_\func]=[\lambda_5]=[0,4]$ for $\hessian(\func(x))$ on $\hyperRec$.
Now, consider the functions $g,h:\R^2 \rightarrow \R$ with $g(x)=x_1^2$ and $h(x)=x_2^2$. From
$$
  \hessian{g}(x)= \left(\begin{array}{cc} 
    2& 0 \\ 
    0& 0 
  \end{array}\right) \quad \text{and} \quad 
  \hessian{h}(x)= \left(\begin{array}{cc} 
    0& 0 \\ 
    0& 2 
  \end{array}\right),  
$$
we infer that both $\hessian{g}(x)$ and $\hessian{h}(x)$ have the 
eigenvalues $0$ and $2$ for every $x \in \hyperRec$. Hence, the eigenvalue bounds $[\lambda_g]=[\lambda_3]=[0,2]$ and $[\lambda_h]=[\lambda_4]=[0,2]$ that result in line 3 and 4 of extended codelist~\eqref{eq:extendedCodelistExample1} are tight.
The eigenvalue bounds $[\lambda_5]=[\lambda_3]+[\lambda_4]=[0,4]$ that result  in the subsequent line are conservative, however. In fact, the Hessian of $\func$ reads
$$
 \hessian{\func}(x)  = \hessian{g}(x)+
  \hessian{h}(x)= \left(\begin{array}{cc} 
    2& 0 \\ 
    0& 2 
  \end{array}\right)  
$$
for all $x \in \hyperRec$ and the tight eigenvalue bounds obviously read $[\lambda_\func^\ast]=[2,2]$.

\end{example}

The Hessian matrices $\hessian{g}(x)$ and  
$\hessian{h}(x)$ in Exmp.~\ref{example:BadEigvalApproxForSum} have zero 
eigenvalues which disappear when adding the two functions to $f(x)= g(x)+ 
h(x)$. The situation illustrated in Example 
\ref{example:BadEigvalApproxForSum} arises naturally  
in the codelists introduced in Section~\ref{sec:oldArithmetic}, because codelists 
build up functions of many variables from functions of very few of 
these variables.  
In order to mitigate eigenvalue bound overestimation in 
these cases, we need to consider functions like $g(x)$ and $h(x)$ in Exmp.~\ref{example:BadEigvalApproxForSum} as functions of only those variables that
they actually depend on nonlinearly. To this end, some simple terminology and intermediate results are introduced in Sect.~\ref{subsec:sparsityHandling}. Subsequently, sparse sums, products, and 
compositions are treated in Sects.~\ref{subsec:sparseBoundsSum}, \ref{subsec:sparseBoundsComposition}, and 
\ref{subsec:sparseBoundsProduct}, respectively.  
Section~\ref{subsec:NumericalComputation} summarizes how to compute the improved eigenvalue bounds based on the rules introduced in Sects.~\ref{subsec:sparseBoundsSum}--\ref{subsec:sparseBoundsProduct}. 

\subsection{Sparsity handling using reduced Hessians and reduced gradients}
\label{subsec:sparsityHandling}

As pointed out in Exmp.~\ref{example:BadEigvalApproxForSum}, sparsity occurs if functions depend at most linearly on some variables $x_i$, where \textit{at most linear dependence} is defined as follows.
\begin{definition}[at most linear dependence] 
Let $f:\openDomain\rightarrow\R$ be a continuously differentiable function on 
an open set $\openDomain \subseteq\R^n$. Let $i\in\setN$. The
function  
$f$ is said to depend at most linearly on $x_i$ if there exists a $c\in \R$ 
such that  
\begin{equation}
\label{eq:atMostLinear}
  \frac{\partial f}{\partial x_i}(x)= c \mbox{ for all } x\in \openDomain. 
\end{equation} 
The function $f$ obviously is independent of $x_i$ if~\eqref{eq:atMostLinear} holds with $c=0$.
\end{definition} 

Assume a function $f$ is known to depend at most linearly on $x_j$ for all $j \in \LinVars_f$, where $\LinVars_f \subseteq \setN$ is a given index set.
Then, only the eigenvalues of the \textit{reduced Hessian} (see Def.~\ref{def:reducedHessian}) associated with the index set $\indexSet = \LinVars^c_f$ are nontrivial, i.e., not necessarily equal to zero.

\begin{definition}[reduced Hessian $\reducedHessian{\indexSet}{f}(x)$] 
\label{def:reducedHessian}
  Let $f:\mathcal{U}\rightarrow\R$ be a twice continuously differentiable function on an 
  open set $\mathcal{U}\subseteq\R^n$. 
  Let $\indexSet \subseteq \setN$ 
  be an index set and let $m=|\indexSet|$.
If $m=0$ set $\reducedHessian{\indexSet}{f}(x) = 0_{0,0}$, otherwise 
   denote the $m$ 
  elements of $\indexSet$ by $j_1<\dots <j_{n-m}$ in  
  ascending order and define the reduced Hessian 
  $\reducedHessian{\indexSet}{f}(x)\in\R^{(n-m)\times(n-m)}$ by its  
  elements  
  \begin{equation*}  
    \left(\reducedHessian{\indexSet}{f(x)}\right)_{ik}= 
    \frac{\partial^2 f(x)}{\partial x_{j_i}\partial x_{j_k}},  
  \end{equation*} 
  where $i,k\in \N_{1,m}$.   
\end{definition} 
 
We also need to consider reduced gradient vectors. 
\begin{definition}[reduced gradient $\reducedGrad{\indexSet}{f}(x)$] 
  \label{def:ReducedGradientWithoutEmbedder} 
  Let $f:U\rightarrow\R$ be a continuously differentiable function on an 
  open set $\openDomain\subseteq\R^n$. 
  Let $\indexSet \subseteq \setN$ 
  be a nonempty index set and let $m=|\indexSet|$.
  Denote the $m$ 
  elements of $\indexSet$ by $j_1<\dots <j_{n-m}$ in  
  ascending order and define the reduced gradient
  $\reducedGrad{\indexSet}{f}(x)\in\R^{n-m}$ by its  
  elements  
  \begin{equation*}  
    \left(\reducedGrad{\indexSet}{f}(x)\right)_{i}= 
    \frac{\partial f(x)}{\partial x_{j_i}}, 
  \end{equation*} 
  where $i\in \N_{1,m}$.   
\end{definition} 

Note that $\indexSet$ may be empty in Def.~\ref{def:reducedHessian}, while there must exist at least one element in $\indexSet$ in Def.~\ref{def:ReducedGradientWithoutEmbedder}.
This difference arises since codelist lines may depend at most linearly on all variables $x_i$ but they are never independent of all $x_i$.

We can easily evaluate eigenvalue bounds for the  Hessian of a function from eigenvalue bounds for its reduced Hessian.
This is stated precisely in Lem.~\ref{lem:fromSparseToFull}. 

\begin{lemma}[spectral bounds for Hessian from reduced Hessian]
\label{lem:fromSparseToFull}
Let $f$ denote a twice continuously differentiable function 
$f:\openDomain\rightarrow \R$ on an open set $\openDomain$. 
Let the index set $\LinVars_f\subseteq \setN$  be such that $f$ depends at most linearly on $x_i$ for all $i \in \LinVars_f$.
Let $\hyperRec \subset \openDomain$ and let the interval $[\lambda_f^\dagger] \subset \R$ be such that
\begin{equation}
\label{eq:fConditionSparseEig}
\lb{\lambda}_f^\dagger \leq \min_{x \in \hyperRec} \lambda_{\min} ( \reducedHessian{\LinVars^c_f}{f(x)} ) \qquad \text{and} \qquad \max_{x \in \hyperRec} \lambda_{\max} ( \reducedHessian{\LinVars^c_f}{f(x)} ) \leq \ub{\lambda}_f^\dagger.
\end{equation} 
Then, the eigenvalues of the Hessian $\hessian f(x)$ on $\hyperRec$ lie in the interval
\begin{equation}
\label{eq:fromSparseToFull}
[\lambda_f] = \left\{ \begin{array}{ll}
\,[\lambda_f^\dagger] & \text{if} \,\, \LinVars_f = \emptyset,\\
\,[0,0] & \text{if} \,\, \LinVars_f = \setN,\\
\,[\min\{\lb{\lambda}_f^\dagger,0 \},\max\{\ub{\lambda}_f^\dagger,0 \}] & \text{otherwise.}
\end{array}\right.
\end{equation}
\end{lemma}

\begin{proof}
We consider the cases in~\eqref{eq:fromSparseToFull} separately. $\LinVars_f = \emptyset$ implies $\LinVars_f^c = \setN$ and consequently $\hessian f(x) = \reducedHessian{\LinVars^c_f}{f(x)}$, which 
proves the first case.
In the second case, i.e., $\LinVars_f = \setN$, we have $\hessian f(x) = \reducedHessian{\LinVars_f}{f(x)}$.
Since $f$ depends at most linearly on $x_i$ for all $i \in \LinVars_f$, we find $\reducedHessian{\LinVars_f}{f(x)}=0_{n,n}$. 
Thus, the eigenvalue bounds $[\lambda_f]=[0,0]$ hold.
Regarding the third case, we note that $\emptyset \subset \LinVars_f \subset \setN$ implies $\emptyset \subset \LinVars_f^c \subset \setN$.
Thus, $m = |\LinVars_f|$ satisfies $0<m<n$. 
Without loss of generality we assume  $\LinVars_f^c=\N_{1,m}$.
Then
$$
\hessian f(x) = \left(\begin{array}{cc} 
\reducedHessian{\LinVars_f^c}{f(x)} & 0_{m,n-m}\\
0_{n-m,m} & \reducedHessian{\LinVars_f}{f(x)}
\end{array}\right) = \left(\begin{array}{cc} 
\reducedHessian{\LinVars_f^c}{f(x)} & 0_{m,n-m}\\
0_{n-m,m} & 0_{n-m,n-m}
\end{array}\right).
$$
Now consider an arbitrary but fixed $x \in \hyperRec$. We obtain
\begin{align}
\label{eg:minHessianFx}
\lambda_{\min} (  \hessian f(x) ) &= \min \{ \lambda_{\min} (  \reducedHessian{\LinVars_f^c}{f(x)} ), 0 \} \quad \text{and} \\
\label{eg:maxHessianFx}
 \lambda_{\max} (  \hessian f(x) ) &= \max \{ \lambda_{\max} (  \reducedHessian{\LinVars_f^c}{f(x)} ), 0 \}
\end{align}
based on the block-diagonal structure of $\reducedHessian{\LinVars_f^c}{f(x)}$.
Bounding~\eqref{eg:minHessianFx} below and bounding~\eqref{eg:maxHessianFx} above for all $x \in \hyperRec$ yields
$$
[\min_{x \in \hyperRec} \lambda_{\min} (  \hessian f(x) ) , \max_{x \in \hyperRec} \lambda_{\max} (  \hessian f(x) )] \subseteq [\min \{ \lb{\lambda}_f^\dagger, 0 \},\max \{ \ub{\lambda}_f^\dagger, 0 \}]
$$
according to Eqs.~\eqref{eg:minHessianFx} and~\eqref{eg:maxHessianFx} and condition~\eqref{eq:fConditionSparseEig}. 
\end{proof}

\subsection{Improved eigenvalue bounds for the sum of two functions}
\label{subsec:sparseBoundsSum}
We collect some recurring conditions first.
\begin{conditions}
\label{cond:ghSum}
Let $g$ and $h$ denote twice continuously differentiable functions 
$g:\openDomain\rightarrow \R$ and   
$h:\openDomain\rightarrow \R$ on an open set $\openDomain \subset \R^n$. 
Let the index sets $\LinVars_g\subseteq \setN$ and $\LinVars_h\subseteq \setN$ be such that $g$ (resp.~$h$) depends at most linearly
on $x_i$ for all $i \in \LinVars_g$ (resp.~all $i \in \LinVars_h$).
Moreover, let the index sets $\IndepOfVars_g \subseteq \LinVars_g$ and $\IndepOfVars_h \subseteq \LinVars_h$ with $\IndepOfVars_g \subset \setN$ and $\IndepOfVars_h \subset \setN$ be such that $g$ (resp.~$h$) is independent of $x_i$ for all $i \in \IndepOfVars_g$ (resp.~all $i \in \IndepOfVars_h$).
Let $\hyperRec \subset \openDomain$ and assume there exist intervals $[\lambda_g^\dagger] \subset \R$ and $[\lambda_h^\dagger] \subset \R$ such that
\begin{align}
\label{eq:gConditionSparseEig}
\lb{\lambda}_g^\dagger \leq \min_{x \in \hyperRec} \lambda_{\min} ( \reducedHessian{\LinVars^c_g}{g(x)} ) \qquad \text{and} \qquad \max_{x \in \hyperRec} \lambda_{\max} ( \reducedHessian{\LinVars^c_g}{g(x)} ) \leq \ub{\lambda}_g^\dagger, \\
\lb{\lambda}_h^\dagger \leq \min_{x \in \hyperRec} \lambda_{\min} ( \reducedHessian{\LinVars^c_h}{h(x)} ) \qquad \text{and} \qquad \max_{x \in \hyperRec} \lambda_{\max} ( \reducedHessian{\LinVars^c_h}{h(x)} ) \leq \ub{\lambda}_h^\dagger.
\end{align} 
\end{conditions}

Now assume Conds.~\ref{cond:ghSum} hold and we intend to calculate eigenvalue bounds for $\hessian f(x)$ on a hyperrectangle $\hyperRec$ for $f(x)= g(x)+ h(x)$. 
We could 
determine eigenvalue bounds for the full Hessians $\hessian g(x)$ and $\hessian h(x)$ with Lem.~\ref{lem:fromSparseToFull} and apply the rule for the eigenvalue bounds of the sum of full Hessians (line {\tt add} in Tab.~\ref{tab:oldArithmetic} reproduced from \cite{Monnigmann2011a}). However, we show in Lem.~\ref{lem:sparseTighterSum} below that it is advantageous to, roughly speaking, carry out calculations with the sparse Hessians as long as possible and to apply Lem.~\ref{lem:fromSparseToFull} as late as possible. 
We first state the rules for determining $\LinVars_f$, $\IndepOfVars_f$ and the eigenvalues of the reduced Hessian of $f$ in Lems.~\ref{lem:fAtMostLinearSum} and \ref{lem:sparseBoundsSum}, respectively. 
The trivial proof of Lemma~\ref{lem:fAtMostLinearSum} is omitted for brevity. 
\begin{lemma}[index sets for sums]
\label{lem:fAtMostLinearSum}
Assume Conds.~\ref{cond:ghSum} hold and consider the function $f:\openDomain \rightarrow \R$ with $f(x)=g(x)+h(x)$.
Let $\LinVars_f = \LinVars_g \cap \LinVars_h$ and $\IndepOfVars_f = \IndepOfVars_g \cap \IndepOfVars_h$.
Then, $f$ depends at most linearly on $x_i$ for all $i \in \LinVars_f$ and $f$ is independent of $x_i$ for all $i \in \IndepOfVars_f$.
\end{lemma}
\begin{lemma}[spectral bounds for reduced Hessian of sums]
\label{lem:sparseBoundsSum}
Assume Conds.~\ref{cond:ghSum} hold and consider the function $f:\openDomain \rightarrow \R$, $f(x)=g(x)+h(x)$.
Let $\LinVars_f = \LinVars_g \cap \LinVars_h$. Then, 
\begin{equation}
\label{eq:sparseBoundsF}
\lb{\lambda}_f^\dagger \leq \min_{x \in \hyperRec} \lambda_{\min} ( \reducedHessian{\LinVars^c_f}{f(x)} ) \qquad \text{and} \qquad \max_{x \in \hyperRec} \lambda_{\max} ( \reducedHessian{\LinVars^c_f}{f(x)} ) \leq \ub{\lambda}_f^\dagger,
\end{equation} 
where $[\lambda_f^\dagger]$ is computed according to the rules listed in Tab.~\ref{tab:CasesReducedBoundsSum}.
\end{lemma}

\begin{table}[htp]
\caption{Rules for the computation of eigenvalue bounds $[\lambda_f^\dagger]$ for the reduced Hessian $\reducedHessian{\LinVars^c_f}{f(x)}$ of a sum $f(x)=g(x)+h(x)$. Let $\LinVars_\cup$ be short for $\LinVars_\cup :=\LinVars_g \cup \LinVars_h$. See the end of Sect. 4.2 for a discussion of the eight cases.}
\label{tab:CasesReducedBoundsSum}
\centering
\small
\begin{tabular}{c|l|l}
case  & $[\lambda_f^\dagger]$ & condition \\
\hline
\hline 
1 & $[0,0]$ & $\LinVars_g = \setN  \wedge \LinVars_h= \setN$   \\
2 & $[\lambda_g^\dagger]$  & $\LinVars_g \subset \setN  \wedge \LinVars_h= \setN$   \\ 
3  & $[\lambda_h^\dagger]$ & $\LinVars_g = \setN  \wedge \LinVars_h \subset  \setN$    \\
4 & $[\min \{\lb{\lambda}_g^\dagger,\lb{\lambda}_h^\dagger\}, \max \{\ub{\lambda}_g^\dagger,\ub{\lambda}_h^\dagger\}]$ & $\LinVars_g \subset \setN  \wedge \LinVars_h \subset \setN \wedge \LinVars_\cup  = \setN$ \\  
5 & $[\lambda_g^\dagger] +[\lambda_h^\dagger]$ & $ \LinVars_\cup \subset \setN \wedge \LinVars_g = \LinVars_h $     \\ 
6 & $[\lambda_g^\dagger]+[\min \{\lb{\lambda}_h^\dagger,0\},\max \{\ub{\lambda}_h^\dagger,0\}]$ & $ \LinVars_\cup \subset \setN \wedge \LinVars_g \subset \LinVars_h $      \\ 
7 & $[\min \{\lb{\lambda}_g^\dagger,0\}, \max \{\ub{\lambda}_g^\dagger,0\}] + [\lambda_h^\dagger]$ & $ \LinVars_\cup \subset \setN \wedge \LinVars_h \subset \LinVars_g $      \\ 
8 &  $[\min \{\lb{\lambda}_g^\dagger,0\}, \max \{\ub{\lambda}_g^\dagger,0\}]+[\min \{\lb{\lambda}_h^\dagger,0\},\max \{\ub{\lambda}_h^\dagger,0\}]$   & $ \LinVars_\cup \subset \setN \wedge  \LinVars_g \nsubseteq 
 \LinVars_h \wedge \LinVars_h \nsubseteq \LinVars_g  $

\end{tabular} 
\end{table}

\begin{proof}
We prove the fourth case in Tab.~\ref{tab:CasesReducedBoundsSum} since it will be instrumental for Exmp.~\ref{exmp:RevisitEx1UsingSparsity}.
All other cases in Tab.~\ref{tab:CasesReducedBoundsSum} can be proven analogously.
The reduced Hessian of $f$ reads $\reducedHessian{\LinVars_f^c}{f(x)}=\reducedHessian{\LinVars_f^c}{g(x)}+\reducedHessian{\LinVars_f^c}{h(x)}$.
From $\LinVars_g \subset \setN$,  $\LinVars_h \subset \setN$, and $\LinVars_\cup = \LinVars_g \cup \LinVars_h  = \setN$, we infer $\emptyset \subset \LinVars_g \subset \setN$ and  $\emptyset \subset \LinVars_h \subset \setN$
and consequently $\emptyset \subset \LinVars_g^c \subset \setN$ and  $\emptyset \subset \LinVars_h^c \subset \setN$.
Thus, the cardinalities $r=|\LinVars_g^c|$ and $s=|\LinVars_h^c|$ satisfy $0<r<n$ and $0<s<n$.
Moreover, $\LinVars_g \cup \LinVars_h = \setN$ implies $\LinVars_g^c \cap \LinVars_h^c  = \emptyset$. Hence, there does not exist any index $i \in \setN$ such that both $i \in \LinVars_g^c $ and $i \in \LinVars_h^c$.
We assume $\LinVars_g^c=\N_{1,r}$ and $\LinVars_h^c = \N_{r+1,r+s}$ without loss of generality.
Note that $\LinVars_f^c = \LinVars_g^c \cup \LinVars_h^c $ implies $\LinVars_f^c = \N_{1,r+s}$ and $m= |\LinVars_f^c|=r+s$ under this assumption.
Thus, $\reducedHessian{\LinVars_f^c}{f(x)}$ equals
\begin{equation}\label{eq:blockMatrixReducedHessianSum}
  \left(\begin{array}{cc} 
    \reducedHessian{\LinVars_g^c}{g(x)} & 0_{r,s}\\
    0_{s,r} & 0_{s,s}
    \end{array}\right)   + \left(\begin{array}{cc} 
    0_{r,r} & 0_{r,s}\\
    0_{s,r} & \reducedHessian{\LinVars_h^c}{h(x)}
    \end{array}\right)  =\left(\begin{array}{cc} 
    \reducedHessian{\LinVars_g^c}{g(x)} & 0_{r,s}\\
    0_{s,r} & \reducedHessian{\LinVars_h^c}{h(x)}
  \end{array}\right).
\end{equation}
The block-diagonal structure implies 
\begin{align}
\label{eg:minHessianFxSum}
\lambda_{\min} (  \reducedHessian{\LinVars_f^c}{f(x)} ) &= \min \{ \lambda_{\min} (  \reducedHessian{\LinVars_g^c}{g(x)} ), \lambda_{\min} (  \reducedHessian{\LinVars_h^c}{h(x)} ) \} \quad \text{and} \\
\label{eg:maxHessianFxSum}
\lambda_{\max} (  \reducedHessian{\LinVars_f^c}{f(x)} ) &= \max \{ \lambda_{\max} (  \reducedHessian{\LinVars_g^c}{g(x)} ), \lambda_{\max} (  \reducedHessian{\LinVars_h^c}{h(x)}) \}
\end{align}
for an arbitrary but fixed $x \in \hyperRec$.
Bounding~\eqref{eg:minHessianFxSum} below and bounding~\eqref{eg:maxHessianFxSum} above for all $x \in \hyperRec$ yields
$$
[\min_{x \in \hyperRec} \lambda_{\min} (  \hessian f(x) ) , \max_{x \in \hyperRec} \lambda_{\max} (  \hessian f(x) )] \subseteq [\min \{ \lb{\lambda}_g^\dagger, \lb{\lambda}_h^\dagger \},\max \{ \ub{\lambda}_g^\dagger, \ub{\lambda}_h^\dagger \}]
$$
where we used Eqs.~\eqref{eg:minHessianFxSum} and~\eqref{eg:maxHessianFxSum} and Conds.~\ref{cond:ghSum}. 
Thus, the eigenvalues of $\reducedHessian{\LinVars_f^c}{f(x)}$ on $\hyperRec$ lie in the interval $[\lambda_f^\dagger]=[\min \{ \lb{\lambda}_g^\dagger, \lb{\lambda}_h^\dagger \},\max \{ \ub{\lambda}_g^\dagger, \ub{\lambda}_h^\dagger \}]$ as claimed in Tab.~\ref{tab:CasesReducedBoundsSum}.
\end{proof}

We anticipated the bounds from Lem.~\ref{lem:sparseBoundsSum} can be shown to be as tight as or tighter than those from the original method proposed in \cite{Monnigmann2011a} that does not account for sparsity. This can now be shown in Lemma~\ref{lem:sparseTighterSum} below. 
Recall the bounds in~\cite{Monnigmann2011a} result in $[\lambda_f] =[\lambda_g] + [\lambda_h]$  for $f(x)=g(x)+h(x)$ 
according to~\cite[Prop 3.2.(iii)]{Monnigmann2011a}.

\begin{lemma}[improved bounds for sums]
\label{lem:sparseTighterSum}
Assume Conds.~\ref{cond:ghSum} hold and let $f$, $[\lambda_f^\dagger]$, and $\LinVars_f$ be as in Lem.~\ref{lem:sparseBoundsSum}. Let $[\lambda_f]$, $[\lambda_g]$, and $[\lambda_h]$ be the eigenvalue bounds for the Hessians $\hessian{f}(x)$, $\hessian{g}(x)$, and $\hessian{h}(x)$ on $\hyperRec$, calculated according to
Eq.~\eqref{eq:fromSparseToFull}. Then,
\begin{equation}
\label{eq:sparseTighterSum}
[\lambda_f] \subseteq [\lambda_g] + [\lambda_h].
\end{equation}
\end{lemma}

\begin{proof}
We prove the relation for the fourth case in Tab.~\ref{tab:CasesReducedBoundsSum}.
The remaining cases can be proven analogously. 
As pointed out in the proof of Lem.~\ref{lem:sparseBoundsSum}, we have $\emptyset \subset \LinVars_g \subset \setN$ and  $\emptyset \subset \LinVars_h \subset \setN$.
Thus, the r.h.s.~in~\eqref{eq:sparseTighterSum} yields
\begin{align}
\nonumber
[\lambda_g] + [\lambda_h] &=[\min \{ \lb{\lambda}_g^\dagger, 0 \},\max \{ \ub{\lambda}_g^\dagger, 0 \} ] + [\min \{ \lb{\lambda}_h^\dagger, 0 \},\max \{ \ub{\lambda}_h^\dagger, 0 \} ] \\
\nonumber
&=  [\min \{ \lb{\lambda}_g^\dagger, 0 \} + \min \{ \lb{\lambda}_h^\dagger, 0 \} , \max \{ \ub{\lambda}_g^\dagger, 0 \}  + \max \{ \ub{\lambda}_h^\dagger, 0 \} ],\\
\label{eq:sparseTighterSumRHS}
&=  [\min \{ \lb{\lambda}_g^\dagger+\lb{\lambda}_h^\dagger, \lb{\lambda}_g^\dagger,\lb{\lambda}_h^\dagger,0 \}  , \max \{ \ub{\lambda}_g^\dagger+\ub{\lambda}_h^\dagger, \ub{\lambda}_g^\dagger,\ub{\lambda}_h^\dagger,0 \} ],
\end{align}
where the equations hold according to the third case in~\eqref{eq:fromSparseToFull}, by definition of the sum of two intervals (see Tab.~\ref{tab:intervalArithmetic}), and by definition of $\min \{ \cdot \}$ and $\max \{ \cdot \}$, respectively.
To evaluate the l.h.s.~in~\eqref{eq:sparseTighterSum}, we have to analyze the index set $\LinVars_f$. 
We obviously have $\LinVars_f = \LinVars_g \cap \LinVars_h \subset \setN$. Thus, the second case in Eq.~\eqref{eq:fromSparseToFull} does not apply.
However, from the conditions characterizing the fourth case in Tab.~\ref{tab:CasesReducedBoundsSum}, it is not clear whether $\LinVars_f = \emptyset$ or $\LinVars_f \supset \emptyset$.
Thus, according to~\eqref{eq:fromSparseToFull}, the l.h.s.~in~\eqref{eq:sparseTighterSum} results in 
$$
[\lambda_f] = \left\{ \begin{array}{ll}
\,[\lambda_f^\dagger] & \text{if} \,\, \LinVars_f = \emptyset,\\
\,[\min \{ \lb{\lambda}_f^\dagger, 0 \},\max \{ \ub{\lambda}_f^\dagger, 0 \} ] & \text{if} \,\, \emptyset \subset \LinVars_f \subset \setN.
\end{array}\right.
$$
However, since $[a] \subseteq [\min \{ \lb{a}, 0 \},\max \{ \ub{a}, 0 \} ]$, the relation
$[\lambda_f] \subseteq [\min \{ \lb{\lambda}_f^\dagger, 0 \},\max \{ \ub{\lambda}_f^\dagger, 0 \} ]$
holds in both cases.
Since $[\lambda_f^\dagger] = [\min \{\lb{\lambda}_g^\dagger,\lb{\lambda}_h^\dagger\}, \max \{\ub{\lambda}_g^\dagger,\ub{\lambda}_h^\dagger\}]$ according to Lem.~\ref{lem:sparseBoundsSum} (resp. Tab.~\ref{tab:CasesReducedBoundsSum}), we obtain
\begin{align}
\nonumber
[\lambda_f] & \subseteq [\min \{ \min \{\lb{\lambda}_g^\dagger, \lb{\lambda}_h^\dagger\}, 0 \}, \max\{ \max \{\ub{\lambda}_g^\dagger, \ub{\lambda}_h^\dagger\}, 0 \}],\\
\label{eq:sparseTighterSumLHS}
& = [\min \{ \lb{\lambda}_g^\dagger, \lb{\lambda}_h^\dagger, 0 \}, \max\{ \ub{\lambda}_g^\dagger, \ub{\lambda}_h^\dagger, 0 \}].
\end{align}
Comparing Eqs.~\eqref{eq:sparseTighterSumRHS} and~\eqref{eq:sparseTighterSumLHS} yields
$$
[\min \{ \lb{\lambda}_g^\dagger, \lb{\lambda}_h^\dagger, 0 \}, \max\{ \ub{\lambda}_g^\dagger, \ub{\lambda}_h^\dagger, 0 \}] \subseteq [\min \{ \lb{\lambda}_g^\dagger+\lb{\lambda}_h^\dagger, \lb{\lambda}_g^\dagger,\lb{\lambda}_h^\dagger,0 \}  , \max \{ \ub{\lambda}_g^\dagger+\ub{\lambda}_h^\dagger, \ub{\lambda}_g^\dagger,\ub{\lambda}_h^\dagger,0 \} ],
$$
which proves~\eqref{eq:sparseTighterSum}. 
\end{proof}

\begin{table}[htp]
\caption{Fundamental cases underlying Tab.~\ref{tab:CasesReducedBoundsSum}.}
\label{tab:BasicCasesSum}
\centering
\small
\begin{tabular}{c|l|l|l}
case  & reduced Hessian $\reducedHessian{\LinVars_f^c}{f(x)}$ &  contribution & condition \\
\hline
\hline 
(i) & $0_{0,0}$  & none  & $\LinVars_g^c = \emptyset \wedge  \LinVars_h^c = \emptyset$ \\
(ii) & $\reducedHessian{\LinVars_f^c}{g(x)}$ & first Hessian & $\LinVars_g^c \supset \emptyset  \wedge  \LinVars_h^c = \emptyset$   \\ 
(iii) & $\reducedHessian{\LinVars_f^c}{h(x)}$ &  second Hessian & $\LinVars_g^c = \emptyset  \wedge  \LinVars_h^c \supset \emptyset$    \\ 
(iv)  & $\reducedHessian{\LinVars_f^c}{g(x)}+\reducedHessian{\LinVars_f^c}{h(x)}$ & both Hessians & $\LinVars_g^c \supset \emptyset  \wedge  \LinVars_h^c \supset \emptyset  $   \\ 
\end{tabular} 
\end{table}

Lemmas~\ref{lem:sparseBoundsSum} and~\ref{lem:sparseTighterSum}
are based on the eight cases listed in Tab.~\ref{tab:CasesReducedBoundsSum}. 
  Since they are not obvious at first sight, it is instructive to see how these eight cases arise from the four simpler ones listed in Tab.~\ref{tab:BasicCasesSum}. 
  The first case in Tab.~\ref{tab:BasicCasesSum} applies if both $\reducedHessian{\LinVars_f^c}{g(x)}$ and $\reducedHessian{\LinVars_f^c}{h(x)}$ vanish because of $\LinVars_g^c= \LinVars_h^c= \emptyset$.
  Since these two conditions, i.e.\ $\LinVars_g^c= \emptyset\wedge \LinVars_h^c= \emptyset$, are equivalent to 
  $\LinVars_g = \setN \wedge \LinVars_h = \setN$, case (i) in Tab.~\ref{tab:BasicCasesSum} is equivalent to case 1 in Tab.~\ref{tab:CasesReducedBoundsSum}. 
    Analogously, cases (ii) and (iii) in Tab.~\ref{tab:BasicCasesSum}, where either $\reducedHessian{\LinVars_f^c}{g(x)}$ or $\reducedHessian{\LinVars_f^c}{h(x)}$ contribute to $\reducedHessian{\LinVars_f^c}{f(x)}$, are equivalent to cases 2 and 3 in Tab.~\ref{tab:CasesReducedBoundsSum}, respectively.
    It remains to relate case (iv) in Tab.~\ref{tab:BasicCasesSum} to cases 4--8 in Tab.~\ref{tab:CasesReducedBoundsSum}.  
  In fact, the conditions of the cases 4--8 in Tab.~\ref{tab:CasesReducedBoundsSum} all imply $\LinVars_g^c \supset \emptyset$ and $\LinVars_h^c \supset \emptyset$,
  which are the defining conditions for case (iv) in Tab.~\ref{tab:BasicCasesSum}. Figure~\ref{fig:fiveSparseCasesSum} illustrates that every instance of case (iv) from Tab.~\ref{tab:BasicCasesSum} actually uniquely belongs to one of the cases 4--8 from Tab.~\ref{tab:CasesReducedBoundsSum}.

\begin{figure}[htp]
\centering
\psfrag{1}[c][c]{\small case 4} \psfrag{2}[c][c]{\small case 5} \psfrag{3}[c][c]{\small case 6}  \psfrag{4}[c][c]{\small case 7} \psfrag{5}[c][c]{\small case 8}
	\includegraphics[width=\linewidth]{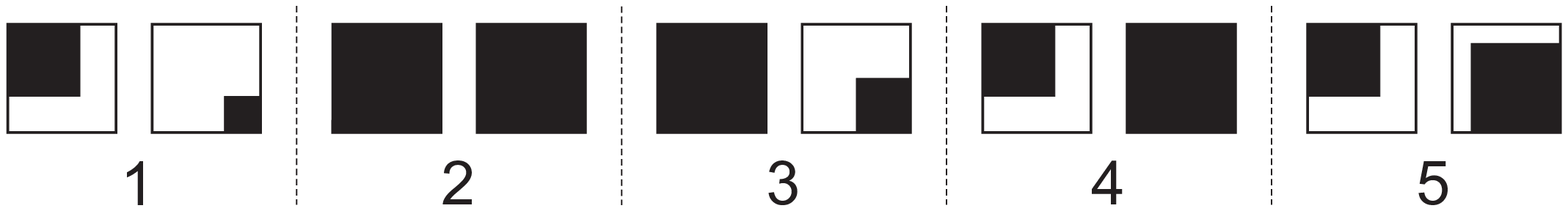} 
		\caption{Illustration of the sparsity patterns of the reduced Hessians $\reducedHessian{\LinVars_f^c}{g(x)}$ and $\reducedHessian{\LinVars_f^c}{h(x)}$ for the cases 4--8 in Tab.~\ref{tab:CasesReducedBoundsSum} and Lem.~\ref{lem:sparseBoundsSum}. White areas correspond to zero blocks in the Hessians, black areas to non-trivial blocks. The black areas in case 4 correspond to the non-trivial submatrices $\reducedHessian{\LinVars_g^c}{g(x)}$ and $\reducedHessian{\LinVars_h^c}{h(x)}$ on the l.h.s.\ of Eq.~\eqref{eq:blockMatrixReducedHessianSum}.}
		\label{fig:fiveSparseCasesSum}
\end{figure}

\subsection{Improved eigenvalue bounds for the composition of two functions}
\label{subsec:sparseBoundsComposition}

We collect some recurring conditions again first.

\MmoThirdRound{Noch Hinweis auf Software aufnehmen und ueberlegen, wie wir die Software anbieten koennen.}
\MmoThirdRound{Alle Lemmas usw.\ mit Kurzbezeichnungen versehen! Lemma~\ref{lem:sparseTighterComposition} in Prop.\ verwandeln, Prop. 4.18 entsprechend in ein Korollar verwandeln?}

\begin{conditions}
\label{cond:gComposition} 
Assume Conds.~\ref{cond:ghSum} hold.
Let $r:\mathcal{V} \rightarrow \R$ be a twice differentiable function on an open set $\mathcal{V} \subset \R$
and assume $g(x) \in \mathcal{V}$ for every $x \in \openDomain$. Moreover, assume there exist intervals $[r^{\prime}] \subset \R$ and $[r^{\prime\prime}] \subset \R$ and a hyperrectangle $[\grad g] \subset \R^n$  such that
\begin{equation}
\label{eq:boundsRandGradientG}
\lb{r}^\prime  \leq r^\prime (g(x))  \leq \ub{r}^\prime, \quad \lb{r}^{\prime\prime}  \leq r^{\prime\prime}  (g(x))  \leq \ub{r}^{\prime\prime}, 
\quad \textrm{and} \quad \lb{\grad g}_i \leq (\grad g(x))_i \leq \ub{\grad g}_i
\end{equation}
for every $x\in \hyperRec$ and every $i \in \setN$, where $r^\prime (z)$ and $r^{\prime\prime}  (z)$ refer to the first and the second derivative of $r(z)$, respectively.
\end{conditions}

The following lemma, which we state without proof, provides rules for the identification of at most linear dependencies and independencies of compositions. 
\begin{lemma}[index sets for compositions]
\label{lem:fAtMostLinearComposition}
Assume Conds.~\ref{cond:gComposition}  hold and consider the function $f:\openDomain \rightarrow \R$, $f(x)=r(g(x))$.
Let
\begin{equation} 
\label{eq:SetIandLComposition}
 \IndepOfVars_f = \IndepOfVars_g \qquad \textrm{and} \qquad \LinVars_f = \left\{\begin{array}{ll} 
\LinVars_g & \textrm{if r is an affine function,} \\
\IndepOfVars_g & \textrm{otherwise.}
\end{array}\right.
\end{equation}
Then, $f$ depends at most linearly on $x_i$ for all $i \in \LinVars_f$ and $f$ is independent of $x_i$ for all $i \in \IndepOfVars_f$.
\end{lemma}

Bounds for compositions can now be calculated as follows. 
\begin{lemma}[spectral bounds for reduced Hessian of compositions]
\label{lem:sparseBoundsComposition}
Assume Conds.~\ref{cond:gComposition} hold and consider the function $f:\openDomain \rightarrow \R$, $f(x)=r(g(x))$.
Let $\LinVars_f$ be defined as in~Lem.~\ref{lem:fAtMostLinearComposition}. Then the bounds~\eqref{eq:sparseBoundsF} hold for
 $[\lambda_f^\dagger]$ computed according to the rules listed in Tab.~\ref{tab:CasesReducedBoundsComposition}.
\end{lemma}

\begin{table}[htp]
\caption{Rules for the computation of eigenvalue bounds $[\lambda_f^\dagger]$ for the reduced Hessian $\reducedHessian{\LinVars^c_f}{f(x)}$ of compositions $f(x)=r(g(x))$.}
\label{tab:CasesReducedBoundsComposition}
\centering
\small
\begin{tabular}{c|l|l}
case   &  $[\lambda_f^\dagger]$ & condition \\
\hline
\hline 
1 & $[r^{\prime\prime}]\,[\Lambda_s([\grad_{\LinVars_f^c} g])]$ & $\LinVars_g = \setN$  \\ 
2  & $[r^{\prime\prime}]\,[\Lambda_s([\grad_{\LinVars_f^c} g])]+[r^\prime]\,[\lambda_g^\dagger]$ & $\LinVars_g \subset \setN \wedge \LinVars_f =  \LinVars_g$ \\ 
3 & $[r^{\prime\prime}]\,[\Lambda_s([\grad_{\LinVars_f^c} g])] +[r^\prime]\,[\min \{ \lb{\lambda}_g^\dagger,0 \}, \max \{ \ub{\lambda}_g^\dagger,0 \}]$  & $\LinVars_g \subset \setN \wedge \LinVars_f \subset  \LinVars_g$ 
\end{tabular} 
\end{table}

\begin{proof}
We prove the last case in Tab.~\ref{tab:CasesReducedBoundsComposition}. The remaining cases can be proven analogously.
  The reduced Hessian of $f$ reads $\reducedHessian{\LinVars_f^c}{f(x)}=r^{\prime\prime}(g(x))\,\reducedGrad{\LinVars_f^c}{g}(x)\, \grad^T_{\LinVars_f^c}g(x)+r^\prime(g(x))\,\reducedHessian{\LinVars_f^c}{g(x)}$.
    Combining the two conditions of case 3 in Tab.~\ref{tab:CasesReducedBoundsComposition} yields $\LinVars_f \subset  \LinVars_g \subset \setN$, which implies $\LinVars_f^c \supset  \LinVars_g^c \supset \emptyset$. Thus, $m=|\LinVars_f^c|$ and $r=|\LinVars_g^c|$ satisfy $m>r>0$. We assume $\LinVars_f^c=\N_{1,m}$ and $\LinVars_g^c = \N_{1,r}$ without loss of generality.
Under these assumptions, we obtain
\begin{equation}
\label{eq:sparseHessianCompositionCase3}
\reducedHessian{\LinVars_f^c}{f(x)}  = r^{\prime\prime}(g(x))\,\reducedGrad{\LinVars_f^c}{g}(x)\, \grad^T_{\LinVars_f^c}g(x)+r^\prime(g(x))\, \left(\begin{array}{cc} 
\reducedHessian{\LinVars_g^c}{g(x)} & 0_{r,s}\\
0_{s,r} & 0_{s,s}
\end{array}\right),
\end{equation}
where $s=m-r>0$. Since $\grad g(x) \in [\grad g]$ for every $x \in \hyperRec$, we find
$$
\lb{\Lambda}_s([\grad_{\LinVars_f^c} g])\leq \lambda_{\min}(\reducedGrad{\LinVars_f^c}{g}(x)\, \grad^T_{\LinVars_f^c}g(x))\quad \textrm{and} \quad \lambda_{\max}(\reducedGrad{\LinVars_f^c}{g}(x)\, \grad^T_{\LinVars_f^c}g(x))  
\leq \ub{\Lambda}_s([\grad_{\LinVars_f^c} g]),
$$
for every $x \in \hyperRec$ according to \cite[Lem. 2.2]{Monnigmann2011a}.
Combining this intermediate result with the bounds on $r^{\prime}(g(x))$ and $r^{\prime\prime}(g(x))$ from Conds.~\ref{cond:gComposition} yields the eigenvalue bounds
$$
[\lambda_f^\dagger] = [r^{\prime\prime}]\,[\Lambda_s([\grad_{\LinVars_f^c} g])] +[r^\prime]\,[\min \{ \lb{\lambda}_g^\dagger,0 \}, \max \{ \ub{\lambda}_g^\dagger,0 \}]
$$
 on $\hyperRec$. \end{proof}

Lemma~\ref{lem:sparseTighterComposition} below shows that the bounds from Lem.~\ref{lem:sparseBoundsComposition} are as tight as or tighter than those from the original method proposed in~\cite{Monnigmann2011a}. Recall the bounds in~\cite{Monnigmann2011a} result in $[\lambda_f] =[r^{\prime\prime}]\,[\Lambda_s([\grad g])] +[r^\prime]\,[\lambda_g]$  for $f(x)=r(g(x))$ 
according to \cite[Prop 3.4]{Monnigmann2011a}.

\MmoSecondRound{Wieso ist das eigentlich hier und an den entsprechenden Stellen f\"ur Summe und Produkt ein Lemma? Sind die Aussagen in den Lemmas wirklich Vorbereitungen f\"ur andere Aussagen? Wenn nicht, m\"ussen wir wohl eine Prop.\ raus machen!
\textcolor{blue}{Ihr Einwand ist genau richtig. Jedoch wuerde ich ihn genau umgekehrt umsetzen. Natuerlich sind alle Aussagen aus den Kap. 4.1 bis 4.4 Vorbereitung auf das Kap. 4.5. Sie haben Recht, dass sich das auch in den Lem. / Prop. widerspiegeln sollte. Ich habe konsequenterweise alle Props. (in Kap. 4.2 bis 4.3) in Lems. umgewandelt. Dazu habe ich in Kap. 4.5 eine Genauigkeitsprop. eingebaut, die auf den Lems. aufbaut.
Nun hat man alle wesentlichen Aussagen in Kap. 4.5 kompakt zusammengefasst: Alg., Genauigkeit, Komplexitaet.}}
\begin{lemma}[improved bounds for compositions]
\label{lem:sparseTighterComposition}
Assume Conds.~\ref{cond:gComposition} hold and let $f$, $[\lambda_f^\dagger]$, and $\LinVars_f$ be as in Lem.~\ref{lem:sparseBoundsComposition}.
Let $[\lambda_f]$ and $[\lambda_g]$ be the eigenvalue bounds for the Hessians $\hessian{f}(x)$ and $\hessian{g}(x)$ on $\hyperRec$, calculated according to
Eq.~\eqref{eq:fromSparseToFull}. Then,
\begin{equation}
\nonumber
[\lambda_f] \subseteq [r^{\prime\prime}]\,[\Lambda_s([\grad g])] +[r^\prime]\,[\lambda_g].
\end{equation}
\end{lemma}

\MmoSecondRound{Wenn wir in Abschnitt~\ref{subsec:sparseBoundsComposition} keine der Props., Lemmas etc. beweisen, stellt sich die Frage, ob es nicht eine viel kompaktere Darstellung gibt, als der Summe, Komposition und dem Produkt jeweils einen eigenen Abschnitt zu widmen. Entweder wir m\"ussen drastisch k\"urzen oder hier mindestens eine Beweisskizze hinschreiben, damit kein Reviewer die K\"urzung mit der Begr\"undung \textit{Bewiesen wird eh nichts mehr.} fordern kann. Der Text der hier steht ist noch zu unverst\"andlich f\"ur eine Beweisskizze. Die Idee zu argumentieren, dass der Beweis wie f\"ur die Summe funktioniert, hier aber andere F\"alle unterschieden m\"ussen, ist aber perfekt f\"ur eine Beweisskizze, die sich andauernd auf die Analogie zur Summe berufen kann.
\textcolor{blue}{Es finden sich doch in Kap. 4.3 und 4.4 einige Beweise. Ich wuerde nicht drastisch kuerzen, da die Struktur dem Leser meines Erachtens hilft. Ich wuerde aber evt. frueh einen Hinweis bringen, dass alles wichtige in Kap. 4.5 steht und alles andere nur darauf vorbereitet.}}
\MmoThirdRound{Noch Vorschlag von msd folgen: \textit{Ich wuerde aber evt. frueh einen Hinweis bringen, dass alles wichtige in Kap. 4.5 steht und alles andere nur darauf vorbereitet.}}

Since the proof is very similar to the proof of Lem.~\ref{lem:sparseTighterSum}, we omit it.

\subsection{Improved eigenvalue bounds for the product of two functions}
\label{subsec:sparseBoundsProduct}

We begin by collecting recurring conditions again.
\begin{conditions}
\label{cond:ghProduct} 
Assume Conds.~\ref{cond:ghSum} hold and assume there exist intervals $[g]\subset \R$ and $[h]\subset \R$ and hyperrectangles $[\grad g] \subset \R^n$ and $[\grad h] \subset \R^n$ such that
$$
\lb{g} \leq g(x) \leq \ub{g},\,\,\,\,\, 
\lb{h} \leq h(x) \leq \ub{h},\,\,\,\,\, 
\lb{\grad g}_i \leq (\grad g(x))_i \leq \ub{\grad g}_i, 
\,\,\,\,  \textrm{and} \,\,\,\,  \lb{\grad h}_i \leq (\grad h(x))_i \leq \ub{\grad h}_i
$$
for every $x\in \hyperRec$ and every $i \in \setN$.
\end{conditions}

The following lemma provides rules for the identification of at most linear dependencies and independences of products.
\begin{lemma}[index sets for products]
\label{lem:fAtMostLinearProduct}
Assume Conds.~\ref{cond:ghProduct} hold and consider the function $f:\openDomain \rightarrow \R$, $f(x)=g(x)\,h(x)$.
Let $\IndepOfVars_f = \IndepOfVars_g \cap \IndepOfVars_h $ and $\LinVars_f = \IndepOfVars_g \cap \IndepOfVars_h $.
Then, $f$ depends at most linearly on $x_i$ for all $i \in \LinVars_f$ and $f$ is independent of $x_i$ for all $i \in \IndepOfVars_f$.
\end{lemma}

Based on  Conds.~\ref{cond:ghProduct}  and Lem.~\ref{lem:fAtMostLinearProduct}, we are able to compute bounds on the spectrum of $\reducedHessian{\LinVars^c_f}{f(x)}$ according to the rules summarized in Lem.~\ref{lem:sparseBoundsProduct} and Tab.~\ref{tab:CasesReducedBoundsProduct}.  As a preparation, we introduce the interval operators
\begin{align}
  \label{eq:LambdaR}   
[\Lambda_r([a],[b])] &:= [\min \{\lb{a},\lb{b}\}, \max \{\ub{a},\ub{b}\} ] \qquad \qquad \qquad \textrm{and} \\
\label{eq:intervalOperatorLambdaStar}
[\Lambda_\star([a],[b],[c])]&:=\frac{1}{2}\,\left[\lb{a}+\lb{b}-\sqrt{(\lb{a}-\lb{b})^2+d},\ub{a}+\ub{b}+\sqrt{(\ub{a}-\ub{b})^2+d}\right]
\end{align}
for real intervals $[a],[b],[c]\subset \R$, where $d = 4\,\max \{\lb{c}^2,\ub{c}^2\}$. Definition~\eqref{eq:LambdaR} is only introduced 
for the sake of a compact notation. Whenever it is more instructive, we use the notation on the r.h.s.\ of \eqref{eq:LambdaR}. 
\begin{lemma}[spectral bounds for reduced Hessian of products]
\label{lem:sparseBoundsProduct}
Assume Conds.~\ref{cond:ghProduct}  hold and consider the function $f:\openDomain \rightarrow \R$ with $f(x)=g(x)\,h(x)$.
Let $\LinVars_f$ be defined as in Lem.~\ref{lem:fAtMostLinearProduct}. Then the bounds~\eqref{eq:sparseBoundsF} hold for 
$[\lambda_f^\dagger]$ computed according to the rules listed in Tab.~\ref{tab:CasesReducedBoundsProduct}.
\end{lemma}

\begin{table}[htp]
\caption{Rules for the computation of eigenvalue bounds $[\lambda_f^\dagger]$ for the reduced Hessian $\reducedHessian{\LinVars^c_f}{f(x)}$ of a product $f(x)=g(x)\,h(x)$. 
The expressions $[\lambda_t]$, $[\lambda_{g,0}]$, $[\lambda_{h,0}]$, $\LinVars_\cup$, and $\LinVars_\cap$ are short for $[\lambda_t]=[\Lambda_t([\grad_{\LinVars_f^c} g],[\grad_{\LinVars_f^c} h])]$, $[\lambda_{g,0}]=[\min \{\lb{\lambda}_g^\dagger,0\}, \max \{\ub{\lambda}_g^\dagger,0\} ]$, $[\lambda_{h,0}]=[\min \{\lb{\lambda}_h^\dagger,0\}, \max \{\ub{\lambda}_h^\dagger,0\} ]$, $\LinVars_\cup = \LinVars_g \cup \LinVars_h$ and $\LinVars_\cap = \LinVars_g \cap \LinVars_h$. Condition $C_\star$ reads  $(\IndepOfVars_g \cup \IndepOfVars_h = \setN) \wedge (|\IndepOfVars_g|=n-1) \wedge (|\IndepOfVars_h|=n-1)$.}
\label{tab:CasesReducedBoundsProduct}
\centering
\small
\begin{tabular}{c@{\,\,\,}ll}
case    & $[\lambda_f^\dagger]$ & condition \\
\hline
\hline 
1 & $[\lambda_t]$ & $\LinVars_g = \setN  \wedge \LinVars_h= \setN$  \\
2 &  $[\lambda_t]+[h]\,[\lambda_g^\dagger]$ & $\LinVars_g \subset \setN  \wedge \LinVars_h= \setN\wedge \LinVars_f = \LinVars_g $ \\ 
3  &  $[\lambda_t]+[h]\,[\lambda_{g,0}]$ & $\LinVars_g \subset \setN  \wedge \LinVars_h= \setN \wedge \LinVars_f \subset \LinVars_g \wedge \neg C_\star $ \\ 
4  &  $[\Lambda_\star([h]\,[\lambda_g^\dagger],[0,0],[\grad_{\IndepOfVars_g^c} g][\grad_{\IndepOfVars_h^c} h])]$ & $\LinVars_g \subset \setN  \wedge \LinVars_h= \setN  \wedge  C_\star $ \\ 
5&  $[\lambda_t]+[g]\,[\lambda_h^\dagger]$ & $\LinVars_g = \setN  \wedge \LinVars_h \subset  \setN  \wedge \LinVars_f = \LinVars_h  $  \\ 
6  &  $[\lambda_t]+[g]\,[\lambda_{h,0}]$ & $\LinVars_g = \setN  \wedge \LinVars_h \subset  \setN \wedge \LinVars_f \subset \LinVars_h \wedge \neg C_\star $ \\ 
7  &  $[\Lambda_\star([0,0],[g]\,[\lambda_h^\dagger],[\grad_{\IndepOfVars_g^c} g][\grad_{\IndepOfVars_h^c} h])]$ & $\LinVars_g = \setN  \wedge \LinVars_h \subset  \setN  \wedge C_\star  $ \\ 
8 &  $[\lambda_t]+[\Lambda_{r}([\Lambda_{r}([h]\,[\lambda_g^\dagger],[g]\,[\lambda_h^\dagger])],[0,0])]$ & $\LinVars_g \subset \setN  \wedge \LinVars_h \subset \setN \wedge \LinVars_\cup = \setN \wedge \LinVars_f \subset \LinVars_\cap $ \\ 
9  &  $[\lambda_t]+[\Lambda_{r}([h]\,[\lambda_g^\dagger],[g]\,[\lambda_h^\dagger])]$  & $\LinVars_g \subset \setN  \wedge \LinVars_h \subset \setN \wedge \LinVars_\cup = \setN \wedge \LinVars_f = \LinVars_\cap \wedge \neg C_\star$ \\
10  &  $[\Lambda_\star([h]\,[\lambda_g^\dagger],[g]\,[\lambda_h^\dagger],[\grad_{\IndepOfVars_g^c} g][\grad_{\IndepOfVars_h^c} h])]$  & $\LinVars_g \subset \setN  \wedge \LinVars_h \subset \setN \wedge   C_\star$ \\  
11   & $[\lambda_t]+[h]\,[\lambda_g^\dagger] +[g]\,[\lambda_h^\dagger]$ & $ \LinVars_\cup \subset \setN   \wedge \LinVars_f = \LinVars_g = \LinVars_h $  \\ 
12   & $[\lambda_t]+[h]\,[\lambda_g^\dagger] +[g]\,[\lambda_{h,0}]$ & $ \LinVars_\cup \subset \setN   \wedge \LinVars_f = \LinVars_g \subset \LinVars_h $ \\ 
13  & $[\lambda_t]+[h]\,[\lambda_{g,0}] +[g]\,[\lambda_h^\dagger]$ & $ \LinVars_\cup \subset \setN  \wedge \LinVars_f = \LinVars_h \subset \LinVars_g $ \\ 
14  & $[\lambda_t]+[\Lambda_r([h]\,[\lambda_g^\dagger] +[g]\,[\lambda_h^\dagger],[0,0])]$ & $ \LinVars_\cup \subset \setN  \wedge \LinVars_f \subset \LinVars_g = \LinVars_h $ \\
15 & $[\lambda_t]+[\Lambda_r([h]\,[\lambda_g^\dagger] +[g]\,[\lambda_{h,0}],[0,0])]$ & $ \LinVars_\cup \subset \setN   \wedge \LinVars_f \subset \LinVars_g \subset \LinVars_h $    \\ 
16   & $[\lambda_t]+[\Lambda_r([h]\,[\lambda_{g,0}] +[g]\,[\lambda_h^\dagger],[0,0])]$ & $ \LinVars_\cup \subset \setN  \wedge \LinVars_f \subset \LinVars_h \subset \LinVars_g $ \\ 
17  & $[\lambda_t]+[h]\,[\lambda_{g,0}] +[g]\,[\lambda_{h,0}]$ & $\LinVars_\cup \subset \setN \wedge  \LinVars_g \nsubseteq \LinVars_h \wedge \LinVars_h \nsubseteq \LinVars_g  $ 
\end{tabular} 
\end{table}

\begin{proof} 
We prove case 10 from Tab.~\ref{tab:CasesReducedBoundsProduct}. Cases 4 and 7 can be shown analogously. The remaining cases can be proven in the same fashion as those treated in the proofs of Lems.~\ref{lem:sparseBoundsSum} and~\ref{lem:sparseBoundsComposition}.
The reduced Hessian of $f$, which reads
\begin{equation}
\label{eq:sparseSumHelper1}
  \reducedHessian{\LinVars_f^c}{f(x)} = \reducedGrad{\LinVars_f^c}{g}(x) \grad^T_{\LinVars_f^c} h(x) + \reducedGrad{\LinVars_f^c}{h}(x) \grad^T_{\LinVars_f^c} g(x) + h(x)\,\reducedHessian{\LinVars_f^c}{g(x)}  \end{equation}
in all cases, is a two-by-two matrix with a particularly simple block structure in case~10. 
To see this, first note that $g$ and $h$ are independent of all but one variable each (the conditions $|\IndepOfVars_g|=n-1$ and $|\IndepOfVars_h|=n-1$ imply $|\IndepOfVars_g^c|=1$ and $|\IndepOfVars_h^c|=1$).
Moreover, $\IndepOfVars_g \cup \IndepOfVars_h = \setN$ implies $\IndepOfVars_g^c \cap \IndepOfVars_h^c = \emptyset$, which implies $g$ and $h$ depend on two different variables. 
Without loss of generality we assume $g$ depends on $x_1$, and $h$ depends on $x_2$, i.e., 
$\IndepOfVars_g^c=\{1\}$ and $\IndepOfVars_h^c = \{2\}$.
As a further preparation note that $\LinVars_f= \IndepOfVars_g \cap \IndepOfVars_h$, which holds according to Lem.~\ref{lem:fAtMostLinearProduct}, implies $\LinVars_f^c= \IndepOfVars_g^c \cup \IndepOfVars_h^c$, which evaluates to $\LinVars_f^c = \{1,2\}$. 
Since $\LinVars_f^c= \{1, 2\}$ and $g$ only depends on $x_1$ (resp.\ $h$ only depends on $x_2$), 
we have 
\begin{equation}
\label{eq:sparseSumHelper2}
\reducedGrad{\LinVars_f^c}{g}(x) =  \begin{pmatrix}
  \frac{\partial}{\partial x_1} g(x) \\
  \frac{\partial}{\partial x_2} g(x)
\end{pmatrix} = \begin{pmatrix}
  \frac{\partial}{\partial x_1} g(x)  \\
  0
\end{pmatrix},
\quad
\reducedHessian{\LinVars_f^c}{g(x)} =  \begin{pmatrix}
  \frac{\partial^2}{\partial x_1^2} g(x) & 0 \\
  0 & 0
\end{pmatrix}
\end{equation}
respectively
\begin{equation}
\label{eq:sparseSumHelper3}
\reducedGrad{\LinVars_f^c}{h}(x) =  \begin{pmatrix}
  \frac{\partial}{\partial x_1} h(x) \\
  \frac{\partial}{\partial x_2} h(x)
\end{pmatrix} = \begin{pmatrix}
  0 \\
  \frac{\partial}{\partial x_2} h(x)  
\end{pmatrix},
\quad
\reducedHessian{\LinVars_f^c}{h(x)} =  \begin{pmatrix}
  0 & 0 \\
  0 & \frac{\partial^2}{\partial x_2^2} h(x)
\end{pmatrix}.
\end{equation}
Substituting \eqref{eq:sparseSumHelper2} and \eqref{eq:sparseSumHelper3} into \eqref{eq:sparseSumHelper1} yields
$$
  \reducedHessian{\LinVars_f^c}{f(x)} =  \begin{pmatrix}
    h(x)\,\frac{\partial^2}{\partial x_1^2}{g(x)} & \frac{\partial}{\partial x_1} g(x)\frac{\partial}{\partial x_2} h(x) \\
    \frac{\partial}{\partial x_1} g(x)\frac{\partial}{\partial x_2} h(x) & g(x)\,\frac{\partial^2}{\partial x_2^2} h(x)
  \end{pmatrix},
$$
where all entries are scalars.
Now, consider the matrix set
$$\mathcal{H} = \{ H \in \R^{2 \times 2} \,|\, H_{11} \in [h]\,[\lambda_g^\dagger], \,H_{22} \in [g]\,[\lambda_h^\dagger],\, H_{12} \in [\grad_{\IndepOfVars_g^c} g][\grad_{\IndepOfVars_h^c} h],\, H=H^T \}$$
and observe $\{ \reducedHessian{\LinVars_f^c}{f(x)}  \in \R^{2 \times 2} \,|\, x \in \hyperRec \} \subseteq  \mathcal{H}$.
To see this, note that $\{ \reducedHessian{\LinVars_g^c}{g(x)}  \in \R  \,|\, x \in \hyperRec \} \subseteq  [\lambda_g^\dagger]$ and 
$\{ \reducedHessian{\LinVars_h^c}{h(x)}  \in \R  \,|\, x \in \hyperRec \} \subseteq  [\lambda_h^\dagger]$, 
since the eigenvalue of a matrix $M\in\R^{1\times 1}$ is $\lambda= M_{1, 1}$. 
According to Lem.~\ref{lem:specialMatrix} stated in the appendix, eigenvalue bounds for the matrix set $\mathcal{H}$ and consequently for $\reducedHessian{\LinVars_f^c}{f(x)}$ on $\hyperRec$ read
$[\Lambda_\star([h]\,[\lambda_g^\dagger],[g]\,[\lambda_h^\dagger],[\grad_{\IndepOfVars_g^c} g][\grad_{\IndepOfVars_h^c} h])]$ as claimed in Tab.~\ref{tab:CasesReducedBoundsProduct}.
\end{proof}

\MmoThirdRound{Immer noch viel Wiederholung in den Teilen zur Summe, zum Komp.\ und zum Produkt bis hin zu einzelnen \"Uberleitungen:}

Lemma~\ref{lem:sparseTighterProduct} shows that the bounds from Lem.~\ref{lem:sparseBoundsProduct} are as tight as or tighter than those from the original method proposed in~\cite{Monnigmann2011a}. Recall the bounds in~\cite{Monnigmann2011a} result in $[\lambda_f] =[\Lambda_t([\grad g],[\grad h])]+[h]\,[\lambda_g] +[g]\,[\lambda_h]$  for $f(x)=g(x)\,h(x)$ according to~\cite[Prop 3.2.(iv)]{Monnigmann2011a}.
We omit the proof of Lem.~\ref{lem:sparseTighterProduct}, since it is similar to its counterparts in Sect.~\ref{subsec:sparseBoundsSum}.

\begin{lemma}[improved bounds for products]
\label{lem:sparseTighterProduct}
Assume Conds.~\ref{cond:gComposition} hold and let $f$, $[\lambda_f^\dagger]$, and $\LinVars_f$ be as in Lem.~\ref{lem:sparseBoundsComposition}.
Let $[\lambda_f]$, $[\lambda_g]$, and $[\lambda_h]$ be the eigenvalue bounds for the Hessians $\hessian{f}(x)$, $\hessian{g}(x)$, and $\hessian{h}(x)$ on $\hyperRec$, calculated according to
Eq.~\eqref{eq:fromSparseToFull}. Then,
$$
[\lambda_f] \subseteq [\Lambda_t([\grad g],[\grad h])]+[h]\,[\lambda_g] +[g]\,[\lambda_h].
$$
\end{lemma}

\subsection{Numerical computation of improved eigenvalue bounds}
\label{subsec:NumericalComputation}

In this section, we combine the results from Sects.~\ref{subsec:sparsityHandling} through~\ref{subsec:sparseBoundsProduct}
in order to compute improved eigenvalue bounds using a codelist.
Formally, this leads to the extended codelist in Prop.~\ref{prop:newArithmetic}.

\begin{proposition}[algorithm for the computation of eigenvalue bounds using sparsity]
\label{prop:newArithmetic}
Assume $\func$ is twice continuously differentiable on $\openDomain$ and can be written as a codelist~\eqref{eq:codelist} with $t \in \N$ operations.
Let $\hyperRec = [x_1] \times \dots \times [x_n] \subset \openDomain$ be arbitrary.
  Then, for all $x\in \hyperRec$, we have $\func(x)\in [\func]$,
  $\grad{\func}(x)\in[\grad{\func}]$, 
  and $[\lambda_{\min}(\hessian{\func}(x)),\lambda_{\max}(\hessian{\func}(x))] \subseteq [\lambda_\func]$, where 
  $[\func]$, $[\grad{\func}]$, and $[\lambda_\func]$ are calculated 
  by the following algorithm.  
  \begin{enumerate} 
  \item For $k= 1, \dots, n$, 
    set $\IndepOfVars_k=\setN\setminus \{k\}$, $\LinVars_k = \setN$, $[y_k]= [\lb{x}_k, \ub{x}_k]$,
    $[\grad y_k]= [e_k,e_k]$, and
    $[\lambda_k^\dagger]= [0, 0]$. 
  \item For $k= n+1, \dots, n+t$, evaluate $\IndepOfVars_k$ and $\LinVars_k$ according to the third and fourth column of Tab.~\ref{tab:indexSets}, respectively.
    Calculate $[y_k]$ and $[\grad y_k]$
    according to the third and fourth column of 
    Tab.~\ref{tab:oldArithmetic}, respectively. Compute $[\lambda_k^\dagger]$ depending on $\LinVars_i$, $\LinVars_j$, and $\LinVars_k$  according to the second column of Tab.~\ref{tab:newArithmetic} in the appendix.  
\item Compute $[\lambda_{n+t}]$ from $[\lambda_{n+t}^\dagger]$ and $\LinVars_{n+t}$ according to Eq.~\eqref{eq:fromSparseToFull} and set $[\func]= [y_{n+t}]$, 
    $[\grad{\func}]= [\grad y_{n+t}]$, 
    and $[\lambda_\func]= [\lambda_{n+t}]$. 
  \end{enumerate} 

\end{proposition}

\begin{table}[htp] 
\caption{Rules for the computation of the sets $\IndepOfVars_k$ and $\LinVars_k$ in the $k$-th line of the codelist \eqref{eq:codelist} for variables and binary operations (left) and compositions (right).
Rules for $y_k$ are repeated here for convenience.}
\label{tab:indexSets}
    \centering
\small
\begin{tabular}{l|l|l|l} 
  {\tt op} $\Phi_k$& $y_k$        & $\IndepOfVars_k$                  
                                  & $\LinVars_k$  \\\hline \hline 
  {\tt var}        & $x_k$        & $\setN\setminus \{k\}$ & $\setN$  \\
\hline   {\tt add}        & $y_i+ y_j$   & $\IndepOfVars_i \cap \IndepOfVars_j$  & $\LinVars_i \cap \LinVars_j$\\ 
  {\tt mul}        & $y_i\,y_j$   & $\IndepOfVars_i \cap \IndepOfVars_j$  & $\IndepOfVars_i \cap \IndepOfVars_j$ \\
  \multicolumn{1}{l}{\textcolor{white}{{\tt a}}} \\ 
  \multicolumn{1}{l}{\textcolor{white}{{\tt a}}} \\ 
    \multicolumn{1}{l}{\textcolor{white}{{\tt a}}} \\ 
      \multicolumn{1}{l}{\textcolor{white}{{\tt a}}} 

\end{tabular} \qquad \quad
\begin{tabular}{l|l|l|l} 
  {\tt op} $\Phi_k$& $y_k$        & $\IndepOfVars_k$                  
                                  & $\LinVars_k$  \\\hline \hline 
  {\tt powNat}     & $y_i^m$      & $\IndepOfVars_i$          & $\IndepOfVars_i$\\ 

  {\tt oneOver}    & $1/y_i$      & $\IndepOfVars_i$          & $\IndepOfVars_i$\\ 

  {\tt sqrt}       & $\sqrt{y_i}$ & $\IndepOfVars_i$          & $\IndepOfVars_i$\\ 
  {\tt exp}        & $\exp(y_i)$  & $\IndepOfVars_i$          & $\IndepOfVars_i$\\ 
  {\tt ln}         & $\ln(y_i)$   & $\IndepOfVars_i$          & $\IndepOfVars_i$\\ 
                                    {\tt addC}   & $y_i+ c$     & $\IndepOfVars_i$          & $\LinVars_i$\\ 
  {\tt mulByC} & $c\,y_i$     & $\IndepOfVars_i$          & $\LinVars_i$ 
\end{tabular} 
\end{table}
\MmoThirdRound{ithenticate laufen lassen. (Done, currently 24\%, which is acceptable, since most of it results from standard statements and citations.)}
\begin{proof}
The claims $\func(x)\in [\func]$ and 
  $\grad{\func}(x)\in[\grad{\func}]$ for all $x\in \hyperRec$ are covered by Thm.~\ref{thm:directArithmetic}.
 It remains to prove that
    $[\lambda_{\min}(\hessian{\func}(x)),\lambda_{\max}(\hessian{\func}(x))] \subseteq [\lambda_\func]$ for all $x\in \hyperRec$. 
Since $y_k(x)=x_k$ for $k=1, \dots, n$, the functions $y_k(x)$, $k \in \setN$, are independent of $x_j$ for every $j\in \IndepOfVars_k=\setN\setminus \{k\}$ and at most linearly dependent on $x_j$ for every  $j \in \LinVars_k = \setN$. 
Thus, the reduced Hessian reads $\reducedHessian{\LinVars_k^c}{y_k(x)}=\reducedHessian{\emptyset}{y_k(x)}= 0_{0,0}$  and  $[\lambda_k^\dagger]=[0,0]$ for every $k \in \setN$.
Now assume eigenvalue bounds
    $[\lambda_1^\dagger], \dots, [\lambda_l^\dagger]$ for the reduced Hessians $\reducedHessian{\LinVars_1^c}{y_1(x)},\dots,\reducedHessian{\LinVars_l^c}{y_l(x)}$ and index sets $\IndepOfVars_1,\dots,\IndepOfVars_l$ and $\LinVars_1,\dots,\LinVars_l$ have been calculated for some     
    $l\in \N_{n,n+t-1}$,
    and let $k= l+ 1$. Since $\Phi_k(y_1,
    \dots, y_{k-1})$ is one of the unary or binary functions listed in
    Tab.~\ref{tab:oldArithmetic} (and therefore Tabs.~\ref{tab:indexSets} and \ref{tab:newArithmetic}), it depends on either one (say $y_i$)
    or two (say $y_i$ and $y_j$) of the intermediate variables 
    $y_1, \dots, y_{k-1}$.
        The remainder of the proof must be carried out for each type of 
    operation $\Phi_k$ separately. 
    We state the proof for one of the {\tt mul} cases and claim the remaining cases can be shown accordingly. Let $g(x)= y_i(x)$,
    $h(x)= y_j(x)$, and $f(x)= y_k(x)$, which implies $f(x)=
      g(x)\,h(x)$, since the operation in the $k$-th line is of type {\tt mul}. 
      In order to compute eigenvalue bounds $[\lambda_f^\dagger]$ for the reduced Hessian  $\reducedHessian{\LinVars_f^c}{f(x)}$, we first evaluate the index sets $\IndepOfVars_f$ and $\LinVars_f$.
      According to Lem.~\ref{lem:fAtMostLinearProduct}, we obtain 
      $$\IndepOfVars_f = \IndepOfVars_g \cap \IndepOfVars_h = \IndepOfVars_i \cap \IndepOfVars_j  \quad \text{and} \quad \LinVars_f = \IndepOfVars_g \cap \IndepOfVars_h = \IndepOfVars_i \cap \IndepOfVars_j, $$
where we used $\IndepOfVars_g = \IndepOfVars_i$ and   $ \IndepOfVars_h = \IndepOfVars_j$,  which hold by construction.    
Assuming we have $\LinVars_g \cup \LinVars_h \subset \setN$ and $\LinVars_f=   \LinVars_g \subset \LinVars_h$,   
applying  Lem.~\ref{lem:sparseBoundsProduct} (specifically, rule 12 in Tab.~\ref{tab:CasesReducedBoundsProduct}) results in 
\begin{equation}
\label{eq:reducedBoundsMulProof}
\begin{array}{r@{\,\,}c@{\,\,}l@{\,}c@{\,}l@{\,}c@{\,}l}
[\lambda_f^\dagger]&=&[\Lambda_t([\grad_{\LinVars_f^c} g],[\grad_{\LinVars_f^c} h])] &+&[h]\,[\lambda_g^\dagger] &+&[g]\,[\min\{\lb{\lambda}_h^\dagger,0\},\max\{\lb{\lambda}_h^\dagger,0\}], \\
&=& [\Lambda_t([\grad_{\LinVars_k^c} y_i],[\grad_{\LinVars_k^c} y_j])]&+&[y_j]\,[\lambda_i^\dagger] &+&[y_i]\,[\min\{\lb{\lambda}_j^\dagger,0\},\max\{\lb{\lambda}_j^\dagger,0\}],
\end{array}
\end{equation}   
where the second equation results from
    substituting the codelist notation $[g]= [y_i]$, $[h]=
    [y_j]$, $[\grad{g}]= [ \grad y_i]$, $[\grad{h}]= [\grad y_j]$,
    $[\lambda_g^\dagger]= [\lambda_i^\dagger]$, and $[\lambda_h^\dagger]= 
    [\lambda_j^\dagger]$. Finally, since $\LinVars_k = \LinVars_f$ and consequently $\reducedHessian{\LinVars_k^c}{y_k(x)}= \reducedHessian{\LinVars_f^c}{f(x)}$,
    the eigenvalues of $\reducedHessian{\LinVars_k^c}{y_k(x)}$ are 
    confined to $[\lambda_k^\dagger]= [\lambda_f^\dagger]$ for all $x\in \hyperRec$. 
    Since the second equation in~\eqref{eq:reducedBoundsMulProof} is equal to the rule
    in Tab.~\ref{tab:newArithmetic} 
    for the case {\tt mul} and $\LinVars_i \cup \LinVars_j \subset \setN$ and $\LinVars_k=   \LinVars_i \subset \LinVars_j$, this proves the claim for the selected case.   
\end{proof}

Proposition~\ref{prop:newArithmetic} is illustrated with two examples.
First, we revisit the motivating Exmp.~\ref{example:BadEigvalApproxForSum}.
Recall that we evaluated the conservative eigenvalue bounds $[\lambda_\func]=[0,4]$ using the original method from~\cite{Monnigmann2011a}.

\begin{example}[improved method applied to  $\func(x)=x_1^2 + x_2^2$ from Exmp.~\ref{example:BadEigvalApproxForSum}]
\label{exmp:RevisitEx1UsingSparsity}
Consider the function $\func$ from Exmp.~\ref{example:BadEigvalApproxForSum} again. 
Proposition~\ref{prop:newArithmetic} results in the following extended codelist. Note that we do not list the expressions for $[y_k]$ and $[\grad y_k]$ in~\eqref{eq:extendedCodelistExample1Sparse} since they are identical to the corresponding expressions in~\eqref{eq:extendedCodelistExample1}. 
Further note that $\IndepOfVars_k$ and $\LinVars_k$ are independent of $\hyperRec$.
\begin{equation}
\label{eq:extendedCodelistExample1Sparse}
\small
\centering
\begin{tabular}{l|l@{\,\,}ll@{\,\,}ll} 
  $k$        &     $\IndepOfVars_k$  &   & $\LinVars_k$ &  & $[\lambda_k^\dagger]$    \\
  \hline 
  \hline 
 $1$        & $\setN\setminus \{1\}$ & $= \{2\}$    & $\setN$ & $= \{1,2\}$ & $[0,0]$ \\
$2$        &  $\setN\setminus \{2\}$ & $= \{1\}$ & $\setN$  & $= \{1,2\}$ & $[0,0]$ \\
\hline

 $3$      & $\IndepOfVars_1$ & $= \{2\}$ & $\IndepOfVars_1$ & $= \{2\}$ & $2\,[\lambdaaaT([\grad_{\LinVars_3^c} y_1])]=2\,[\lambdaaaT([\grad_{\{1\}} y_1])]$  \\ 
                                  
$4$      & $\IndepOfVars_2$  & $= \{1\}$ & $\IndepOfVars_2$ & $= \{1\}$ & $2\,[\lambdaaaT([\grad_{\LinVars_4^c} y_2])]=2\,[\lambdaaaT([\grad_{\{2\}} y_2])]$   \\ 
                                  
$5$   & $\IndepOfVars_3 \cap \IndepOfVars_4$ & $= \emptyset$ & $\LinVars_3 \cap \LinVars_4$ & $= \emptyset$ & $[\min\{\lb{\lambda}_3^\dagger,\lb{\lambda}_4^\dagger\},\max\{\ub{\lambda}_3^\dagger,\ub{\lambda}_4^\dagger\}]$   \\ 
\hline 
&  &  & & & $[\lambda_\varphi] = [\lambda_5^\dagger]$
\end{tabular}
\end{equation}
The expressions for $[\lambda_k^\dagger]$ in lines $3$ and $4$ of the extended codelist in~\eqref{eq:extendedCodelistExample1Sparse} refer to the first rule associated with the  {\normalfont \texttt{powNat}}-operation in Tab.~\ref{tab:newArithmetic} since $\LinVars_1=\setN$ and $\LinVars_2=\setN$, respectively.
Since  $\LinVars_3 = \{2\} \subset \setN$, $\LinVars_4 = \{1\} \subset \setN$, and $\LinVars_3 \cup \LinVars_4 = \setN$, we
 obtain the bounds $[\lambda_5^\dagger]=[\min\{\lb{\lambda}_3^\dagger,\lb{\lambda}_4^\dagger\},\max\{\ub{\lambda}_3^\dagger,\ub{\lambda}_4^\dagger\}]$ according to the last rule for the {\normalfont \texttt{add}}-operation  in  Tab.~\ref{tab:newArithmetic}.
Finally, since $\LinVars_5 = \emptyset$, we have $[\lambda_\func]=[\lambda_5]=[\lambda_5^\dagger]$ according to Eq.~\eqref{eq:fromSparseToFull}.

Evaluating the extended codelist~\eqref{eq:extendedCodelistExample1Sparse} for the hyperrectangle $\hyperRec = [0,1] \times [0,1]$ (as in Exmp.~\ref{example:BadEigvalApproxForSum}) by computing 
 $[y_k]$ and $[\grad y_k]$ according to~\eqref{eq:extendedCodelistExample1} and $[\lambda_k]$ according to~\eqref{eq:extendedCodelistExample1Sparse} yields
 $[\lambda_1^\dagger]=[\lambda_2^\dagger]=[0,0]$ and $[\lambda_3^\dagger]=[\lambda_4^\dagger]=[\lambda_5^\dagger]=[2,2]$, 
 where we used $[\lambdaaaT([\grad_{\{1\}} y_1])]=[\lambdaaaT([1,1])]=[1,1]$ and $[\lambdaaaT([\grad_{\{2\}} y_2])]=[\lambdaaaT([1,1])]=[1,1]$ (see Eq.~\eqref{eq:extendedCodelistExample1Results} for numerical results on $[y_k]$ and $[\grad y_k]$).
 Thus, using the improved method, we obtain the tight eigenvalue bounds $[\lambda_\func]=[\lambda_5^\dagger]=[\lambda_\func^\ast]=[2,2]$.
\end{example}

We analyze another example to demonstrate that the new method results in considerable improvements for all functions that involve multiplications. 
In fact, we know from \cite[Rem. 4.3]{Monnigmann2011a} that $0 \in [\lambda_\func]$ for the original method if the {\normalfont\texttt{mul}}-operation is required in the codelist of any $\func$ with $n \geq 2$. This is a severe drawback of the original method, 
since it implies that any convex (resp. concave) function $\func:\R^n \rightarrow \R$ involving {\normalfont\texttt{mul}}-operations will never be identified to be convex (resp. concave) using the method from~\cite{Monnigmann2011a}. The following example shows that this restriction does not apply for the improved method.

\begin{example}[comparison of \cite{Monnigmann2011a} and improved method for $\func(x)=x_1^2 + x_2 \exp(x_2)$]
\label{example:mulSparse}
Consider the function $\func:\R^2 \rightarrow \R$ with $\func(x)=x_1^2 + x_2 \exp(x_2)$ on a $\hyperRec\subset \R^2$.
Theorem~\ref{thm:directArithmetic} (i.e., the original method from \cite{Monnigmann2011a}) results in the following extended codelist, where the expressions for $y_k$ are only listed for illustration of the codelist~\eqref{eq:codelist} of $\func$.
We skip the first three lines, since they are identical to those in~\eqref{eq:extendedCodelistExample1}.
\begin{equation}
\label{eq:extendedCodelistExample3}
\centering
\small
\!\begin{tabular}{l@{\,\,}|l@{\,\,}|l@{\,\,}|l@{\,\,}|l@{\,}} 
 $k$ &  $y_k$        &  $[y_k]$   &  $[\grad y_k]$  &  $[\lambda_k]$ \\\hline \hline 
                    
4 & $\exp(y_2)$     & $[\exp([y_2])]$        
                                  & $[y_4]\,[\grad y_2]$ 
                                  & $[y_4]\,([\lambdaaaT([\grad y_2])]+[\lambda_2])$  \\ 
                                  
5 & $y_2\,y_4$    & $[y_2]\,[y_4]$        
                                  & $[y_4] [\grad y_2]\!+\! [y_2] [\grad y_4]$ 
                                  & $[y_4] [\lambda_2]\!+ \![y_2] [\lambda_4]\!+\! [\lambdaabba([\grad y_2], [\grad y_4])]$ \\ 
6 & $y_3+y_5$    & $[y_3]+ [y_5]$        
                                  & $[\grad y_3]+ [\grad y_5]$ 
                                  & $[\lambda_3]+ [\lambda_5]$ \\ 
                                  \hline
                                  & $\func = y_6$ & $[\func]=[y_6]$ & $[\grad \func]=[\grad y_6]$ & $[\lambda_\func]=[\lambda_6]$  
\end{tabular}\!\!\!
\end{equation}
Evaluating this codelist for $\hyperRec = [0,1] \times [0,1]$ yields
$$[\lambda_\func]=[\lambda_6]=[-\exp(1)+1,3\exp(1)+2]\approx [-1.7183,10.1548].$$

Proposition~\ref{prop:newArithmetic} (i.e., the improved method) results in the following extended codelist. 
The first three lines are identical to those in~\eqref{eq:extendedCodelistExample1Sparse} in this case. 
Further note that the expressions for $[y_k]$ and $[\grad y_k]$ can be found in Eq.~\eqref{eq:extendedCodelistExample1} (lines 1-3) and Eq.~\eqref{eq:extendedCodelistExample3} (lines 4-6).
\begin{equation}
\label{eq:extendedCodelistExample3Sparse}
\centering
\small
\begin{tabular}{l@{\,\,}|l@{\,\,}ll@{\,\,}l@{\,\,}|l} 
  $k$        &     $\IndepOfVars_k$    &       
                                  & $\LinVars_k$  &
                                  & $[\lambda_k^\dagger]$ \\\hline \hline 
                                  
$4$      & $\IndepOfVars_2$  &  $= \{1\}$       
                                  & $\IndepOfVars_2$  &  $= \{1\}$  
                                  & $[y_4]\,[\lambdaaaT([\grad_{\LinVars_4^c} y_2])]=[y_4]\,[\lambdaaaT([\grad_{\{2\}} y_2])]$  \\ 
                                  
$5$   & $\IndepOfVars_2 \cap \IndepOfVars_4$  &  $= \{1\}$       
                                  & $\IndepOfVars_2 \cap \IndepOfVars_4$ & $= \{1\}$  
                                  & $[\Lambda_t([\grad_{\LinVars_5^c} y_2],[\grad_{\LinVars_5^c} y_4])]\!+\![y_4]\,[\lambda_2^\dagger]\!+\![y_2]\,[\lambda_4^\dagger]$  \\    
  &  &   &  &   & $=[\Lambda_t([\grad_{\{2\}} y_2],[\grad_{\{2\}} y_4])]\!+\![y_4]\,[\lambda_2^\dagger]\!+\![y_2]\,[\lambda_4^\dagger]$  \\   
$6$   & $\IndepOfVars_3 \cap \IndepOfVars_5$  &  $= \emptyset$       
                                  & $\LinVars_3 \cap \LinVars_5$ & $= \emptyset$  
                                  & $[\min\{\lb{\lambda}_3^\dagger,\lb{\lambda}_5^\dagger\},\max\{\ub{\lambda}_3^\dagger,\ub{\lambda}_5^\dagger\}]$   \\ 
\hline 
&  &  & & & $[\lambda_\varphi] = [\lambda_6^\dagger]$
\end{tabular}\!\!\!
\end{equation}
Evaluating~\eqref{eq:extendedCodelistExample3Sparse} for $\hyperRec = [0,1] \times [0,1]$ yields
$$[\lambda_\func]=[\lambda_6^\dagger]=[2,3\exp(1)]\approx [2,8.1548],$$
Just as in Exmp.~\ref{example:mulSparse}, the improved method results in tight spectral bounds while the original method from~\cite{Monnigmann2011a} provides loose outer approximations.
In particular, $0 \in [-1.7183,10.1548]$ for the original method as predicted by \cite[Rem. 4.3]{Monnigmann2011a} but $0 \notin [2,8.1548]$ for the improved method presented here. Convexity of $\varphi$ on $\hyperRec$ can therefore be established with the improved but not with the original method.
\end{example}

More generally, the improved method results in 
eigenvalue bounds that are always as tight as, or tighter than, the original method from~\cite{Monnigmann2011a}, as stated in the following proposition.

\begin{proposition}[accuracy of the improved method]
\label{prop:accuracy}
Assume $\func$ is twice continuously differentiable on $\openDomain$ and can be written as a codelist~\eqref{eq:codelist}.
Let $\hyperRec = [x_1] \times \dots \times [x_n] \subset \openDomain$ be arbitrary and denote the eigenvalue bounds for $\hessian{\func}(x)$ on~$\hyperRec$ computed according to Thm.~\ref{prop:newArithmetic} and Prop.~\ref{prop:newArithmetic} by
 $[\lambda_\func^{(\normalfont\text{Thm.~\ref{thm:directArithmetic}})}]$ and $[\lambda_\func^{(\normalfont\text{Prop.~\ref{prop:newArithmetic}})}]$, respectively.
Then, $$[\lambda_\func^{(\normalfont\text{Prop.~\ref{prop:newArithmetic}})}] \subseteq [\lambda_\func^{(\normalfont\text{Thm.~\ref{thm:directArithmetic}})}].$$
\end{proposition} 
\begin{proof}
The proof immediately follows from Lems.~\ref{lem:sparseTighterSum}, \ref{lem:sparseBoundsComposition}, and ~\ref{lem:sparseBoundsProduct}.
\end{proof}

In~\cite[Prop. 4.4]{Monnigmann2011a} it was shown that the numerical complexity for evaluating the extended codelist resulting from Thm.~\ref{thm:directArithmetic},
is of order $\Order(n)\,N(\func)$, where $N(\func)$ denotes the number of operations needed to evaluate $\func$ at a point in its domain.
It is remarkable that this order of complexity can be maintained for the improved method. This is summarized in the following proposition.

\begin{proposition}[numerical complexity of the improved method]
\label{prop:numericalComplexity}
Assume $\func$ is twice continuously differentiable on $\openDomain$ and can be written as a codelist~\eqref{eq:codelist} with $t = N(\func)$ operations.
Let $N ([\func], [\grad \func], [\lambda_\func])$ denote the
number of operations that are necessary to calculate the bounds $[\func] \subset \R$, $[\grad \func] \subset \R^n$, and
 $[\lambda_\func] \subset \R$ for a given hyperrectangle $\hyperRec \subset \openDomain$ using the extended codelist from Prop.~\ref{prop:newArithmetic}.
Then, $$N ([\func], [\grad \func], [\lambda_\func]) = \Order(n)\,N(\func).$$
\end{proposition}

Since the proof of Prop.~\ref{prop:numericalComplexity} is very similar to that of \cite[Prop. 4.4]{Monnigmann2011a}, we only sketch it. 
The extended codelist that results from Prop.~\ref{prop:newArithmetic} involves the index sets $\IndepOfVars_k$ and $\LinVars_k$, which were not required in the original method. These index sets do not depend on the particular hyperrectangle~$\hyperRec$ as pointed out in Exmp.~\ref{exmp:RevisitEx1UsingSparsity}, but they are uniquely determined by the function $\func$ itself. Consequently, all index sets need to be determined only once. This step can be carried out at the time of construction of the extended codelist. In particular, it need not be repeated at the time of evaluating the codelist for a particular $\hyperRec$. 
Once $\IndepOfVars_k$ and $\LinVars_k$ have been determined, each line of the extended codelist is specified by the rules in Tab.~\ref{tab:newArithmetic} (and Tab.~\ref{tab:oldArithmetic}).
It is easy to show that the evaluation of every expression in the second column of Tab.~\ref{tab:newArithmetic} requires at most $\Order(n)$ basic operations (like additions, multiplications, or comparisons of two real numbers; see \cite[Sect. 4.1]{Monnigmann2011a} for further details). Thus, under the assumption that $[\func]$ and $[\grad \func]$ are known, we need $\Order(n)\,N(\func)$ basic operations for the computation of $[\lambda_\func]$. Since the calculation of $[\func]$ and $[\grad \func]$ require $\Order(1)$ and $\Order(n)$ basic operations according to standard results in AD and IA (see, e.g., \cite{Fischer1995,Kearfott1996}), we obtain $N ([\func], [\grad \func], [\lambda_\func]) = \Order(1)\,N(\func)+\Order(n)\,N(\func)+\Order(n)\,N(\func)=\Order(n)\,N(\func).$

\section{Numerical experiments for a large number of examples}
\label{sec:Examples}

In this section, we analyze 1522 numerical examples taken from the COCONUT collection of optimization problems \cite{Shcherbina2003}.
We consider all COCONUT problems with $1<n\leq 10$ variables and extract those cost and constraint functions that can be decomposed into the operations listed 
in Tabs.~\ref{tab:oldArithmetic} and~\ref{tab:newArithmetic}. For each function $\func: \R^n \rightarrow \R$, we consider 100 (randomly generated) hyperrectangles $\hyperRec \subset \mathcal{D}$ in the domain $\mathcal{D}$ of $\func$ specified in the respective COCONUT problem. 
For ease of comparison, the set of examples as well as the associated hyperrectangles are identical to the examples considered in \cite{SchulzeDarup2014JGO}.

For each of the resulting $1522 \cdot 100$ sample problems, we solve problem~\eqref{eq:EigenboundsOnBox} using the improved algorithm  ($\ANew$ for short) in Prop.~\ref{prop:newArithmetic}.
We compare the resulting eigenvalue bounds with those obtained from two established methods using interval Hessians (see problem~\eqref{eq:symIntervalMatrix}) and either Gershgorin's circle criterion ($\Ge$ for short) or Hertz and Rohn's method ($\HR$ for short) for the computation of spectral bounds of interval matrices (see problem~\eqref{eq:EigenboundsSymIntervalMatrix}). 
We choose $\Ge$ and $\HR$ as reference procedures due to the favorable computational complexity of $\Ge$ and since $\HR$ provides tight eigenvalue bounds for problem~\eqref{eq:EigenboundsSymIntervalMatrix} (cf. Sect.~\ref{sec:Introduction}).  
We refer to the original papers \cite{Adjiman1998a,Gershgorin1931,Hertz1992,Rohn1994} or the summaries in~\cite{Monnigmann2011a,SchulzeDarup2014JGO} for a detailed description of methods $\Ge$ and $\HR$.

\begin{table}[h] 
\centering 
\small
\caption{Classes used to aggregate results in Tab.~\ref{tab:dimensionTable}. Symbols $[\lambda_\ANew]=[\lb{\lambda}_\ANew,\ub{\lambda}_\ANew]$, $[\lambda_\Ge]=[\lb{\lambda}_\Ge,\ub{\lambda}_\Ge]$,
and $[\lambda_\HR]=[\lb{\lambda}_\HR,\ub{\lambda}_\HR]$ denote eigenvalue bounds calculated by the improved algorithm $\ANew$ (see Prop.~\ref{prop:newArithmetic}), Gershgorin's circle criterion and Hertz and Rohn's method, respectively.}
\label{tab:casesForComparison}
\begin{tabular}{c@{\,\,}|l|l|l}
class & verbal definition & \multicolumn{2}{c}{formal definition}\\
 & & lower bound $\lb{\lambda}$ & upper bound $\ub{\lambda}$\\
\hline
\hline
1 & $\ANew$ worse than $\Ge$ (and $\HR$) & $\lb{\lambda}_\ANew < \lb{\lambda}_\Ge \leq \lb{\lambda}_\HR$& $\ub{\lambda}_\HR \leq \ub{\lambda}_\Ge < \ub{\lambda}_\ANew$ \\
2 & $\ANew$ equal to $\Ge$ but worse than $\HR$ & $\lb{\lambda}_\Ge \approx \lb{\lambda}_\ANew < \lb{\lambda}_\HR$ & $\ub{\lambda}_\HR < \ub{\lambda}_\ANew \approx \ub{\lambda}_\Ge $ \\
3 & $\ANew$ better than $\Ge$  but worse than $\HR$ &  $\lb{\lambda}_\Ge < \lb{\lambda}_\ANew < \lb{\lambda}_\HR$ & $\ub{\lambda}_\HR < \ub{\lambda}_\ANew < \ub{\lambda}_\Ge$ \\
4 & $\ANew$ equal to $\HR$ (and equal to or better than $\Ge$) & $\lb{\lambda}_\Ge \leq \lb{\lambda}_\HR \approx \lb{\lambda}_\ANew$ & $\ub{\lambda}_\ANew \approx \ub{\lambda}_\HR \leq \ub{\lambda}_\Ge $ \\
5 & $\ANew$ better than $\HR$ (and $\Ge$) & $\lb{\lambda}_\Ge \leq \lb{\lambda}_\HR < \lb{\lambda}_\ANew$ & $\ub{\lambda}_\ANew < \ub{\lambda}_\HR \leq \ub{\lambda}_\Ge$ \\
\end{tabular}
\end{table} 

For each sample problem, we analyze whether $\ANew$ performs better than, equally good as, or worse than $\Ge$ and $\HR$.
We independently compare the lower and upper eigenvalue bounds of the particular methods and 
categorize the results according to the five classes in Tab.~\ref{tab:casesForComparison}.
Note that $\Ge$ never performs better than $\HR$ (since $\HR$ provides tight bounds for~\eqref{eq:EigenboundsSymIntervalMatrix}). 
Consequently, the relations $\lb{\lambda}_\Ge \leq \lb{\lambda}_\HR$ and $\ub{\lambda}_\HR \leq \ub{\lambda}_\Ge$ always hold.
Hence, the list of classes in Tab.~\ref{tab:casesForComparison} is complete in the sense that every example can be uniquely classified into one of the five classes.
It remains to comment on the precise meaning of $a>b$ and $a \approx b$ as used for the classification in Tab.~\ref{tab:casesForComparison}.
To this end, we introduce the function
$$
\dev(a,b) = \frac{a-b}{1+0.5\,|a+b|}
$$
which evaluates a weighted difference of $a,b \in \R$. Based on $\dev(a,b)$, we define
\begin{equation}
\label{eq:ComparisonEps}
a > b \quad  \Longleftrightarrow \quad \dev(a,b) > \epsilon \quad \text{and} \quad a \approx b \quad  \Longleftrightarrow \quad |\dev(a,b)| \leq \epsilon,
\end{equation}
where $\epsilon \in \R_+$ represents an error bound.
Note that $|\dev(a,b)|$ is approximately equal to the relative error for two large but almost equal numbers $a,b \in \R$ and almost equal to the absolute error for two small but almost equal numbers $a,b \in \R$.
This behavior is useful since the absolute values of the computed eigenvalue bounds range across multiple magnitudes. 
\begin{table}[htp] 
\caption{Numerical results for $1522 \cdot 100$ sample problems. The classification of the results is carried out according to Tab.~\ref{tab:casesForComparison} and Eq.~\eqref{eq:ComparisonEps} with $\epsilon = 10^{-6}$. For each class, the left respectively right column refers to eigenvalue bounds obtained by the original algorithm $\AOld$ from~\cite{Monnigmann2011a} and the improved variant $\ANew$ taking sparsity into account.} 
\label{tab:dimensionTable} 
\centering 
\begin{tabular}{r@{\,\,}r|rr|rr|rr|rr|rr} 
\multicolumn{2}{c|}{examples}  & \multicolumn{2}{c|}{$1$} & \multicolumn{2}{c|}{$2$} & \multicolumn{2}{c|}{$3$} & \multicolumn{2}{c|}{$4$} & \multicolumn{2}{c}{$5$} \\ 
$n$ & $\#$ & \multicolumn{1}{c}{$\AOld$} & \multicolumn{1}{c|}{$\ANew$} & \multicolumn{1}{c}{$\AOld$} & \multicolumn{1}{c|}{$\ANew$} & \multicolumn{1}{c}{$\AOld$} & \multicolumn{1}{c|}{$\ANew$} & \multicolumn{1}{c}{$\AOld$} & \multicolumn{1}{c|}{$\ANew$} & \multicolumn{1}{c}{$\AOld$} & \multicolumn{1}{c}{$\ANew$} \\
\hline 
\hline 
2 & 62& 56.37& 29.06& 1.07& 1.08& 12.24& 16.06& 18.89& 41.27& 11.43& 12.54\\
3 & 1078& 2.98& 0.95& 77.53& 0.31& 0.85& 1.01& 17.70& 96.72& 0.95& 1.01\\
4 & 67& 60.75& 35.28& 4.84& 4.63& 8.45& 16.76& 15.04& 32.14& 10.92& 11.19\\
5 & 88& 56.80& 35.34& 3.32& 0.21& 10.16& 14.90& 14.91& 34.47& 14.81& 15.09\\
6 & 95& 35.05& 25.65& 5.31& 4.13& 31.91& 34.25& 13.87& 21.38& 13.86& 14.58\\
7 & 27& 65.80& 27.07& 11.93& 8.50& 0.02& 19.44& 22.26& 44.98& 0.00& 0.00\\
8 & 15& 94.23& 63.83& 3.20& 4.23& 1.27& 29.80& 1.30& 1.53& 0.00& 0.60\\
9 & 24& 57.27& 25.02& 8.29& 0.02& 18.71& 34.94& 4.25& 22.58& 11.48& 17.44\\
10 & 66& 51.17& 12.24& 1.00& 0.22& 22.12& 34.11& 15.86& 41.54& 9.86& 11.89\\
\hline 
all & 1522& 17.77& 9.09& 56.11& 0.94& 5.32& 7.78& 16.86& 77.90& 3.95& 4.30
\end{tabular} 
\end{table} 

We summarize numerical results for the analyzed examples in Tab.~\ref{tab:dimensionTable} (with $\epsilon = 10^{-6}$).
We list the percentage of samples that fall into the classes 1 to 5 from Tab.~\ref{tab:casesForComparison} separated by dimension $n$ of the underlying example.
In order to compare the improved algorithm in Prop.~\ref{prop:newArithmetic} to the original method from~\cite{Monnigmann2011a} (see Thm.~\ref{thm:directArithmetic}), 
we also list the classification results using the original algorithm ($\AOld$ for short).
The numerical results confirm that the consideration of sparsity significantly improves the tightness of the computed eigenvalue bounds. 
To see this, note that for each dimension $n$, the percentages in class 1 (where the established approaches outperform the direct computation of eigenvalue bounds) decrease while the percentages in classes 4 and 5 (where the direct computation of eigenvalue bounds performs as good as or better than Hertz and Rohn's method) increases using the improved algorithm $\ANew$ instead of the original $\AOld$.
In particular, it is remarkable that the improved algorithm $\ANew$ results in worse eigenvalue bounds than $\Ge$ in only $9.09\,\%$ of all cases in contrast to 
$17.77\,\%$ for the original method $\AOld$. Moreover, $\ANew$ provides equally good or better eigenvalue bounds than $\HR$ in $82.20\,\% = 77.90\,\% +4.30\,\%$ of all cases while the corresponding percentage only reads $20.83\,\% = 16.86\,\% +3.95\,\%$ for $\AOld$.

Another observation is that the ratios in the particular classes seem to be independent of the dimension $n$ (i.e., there is no trend).
This is important since the numerical complexities of the established approaches $\Ge$ and $\HR$ 
vary between $\Order(n)\,N(\func)+\Order(n^2)$ and $\Order(n^2)\,N(\func)+ \Order(2^{n}\,n^3)$ operations (see Sect.~\ref{sec:Introduction} and the benchmark in \cite{SchulzeDarup2014JGO}), while the direct eigenvalue bound computation requires $\Order(n)\,N(\func)$.
Thus, methods $\AOld$ and $\ANew$ become numerically very attractive for high dimensions $n$.

\begin{table}[h] 
\caption{The last line of Tab.~\ref{tab:dimensionTable} for different choices of the error bound $\epsilon$ in~\eqref{eq:ComparisonEps}.} 
\label{tab:epsTable} 
\centering 
\begin{tabular}{c|rr|rr|rr|rr|rr} 
  & \multicolumn{2}{c|}{$1$} & \multicolumn{2}{c|}{$2$} & \multicolumn{2}{c|}{$3$} & \multicolumn{2}{c|}{$4$} & \multicolumn{2}{c}{$5$} \\ 
$\epsilon$ & \multicolumn{1}{|c}{$\AOld$} & \multicolumn{1}{c|}{$\ANew$} & \multicolumn{1}{|c}{$\AOld$} & \multicolumn{1}{c|}{$\ANew$} & \multicolumn{1}{|c}{$\AOld$} & \multicolumn{1}{c|}{$\ANew$} & \multicolumn{1}{|c}{$\AOld$} & \multicolumn{1}{c|}{$\ANew$} & \multicolumn{1}{|c}{$\AOld$} & \multicolumn{1}{c}{$\ANew$} \\
\hline 
\hline
$10^{-5}$ &  17.16& 8.78& 54.11& 0.91& 4.29& 6.57& 20.78& 79.74& 3.66& 4.01 \\
$10^{-6}$ &  17.77& 9.09& 56.11& 0.94& 5.32& 7.78& 16.86& 77.90& 3.95& 4.30 \\
$10^{-7}$ & 18.04& 9.28& 57.62& 0.66& 5.61& 8.24& 14.67& 77.41& 4.06& 4.41
\end{tabular} 
\end{table}

According to~\eqref{eq:ComparisonEps}, the classification in Tab.~\ref{tab:dimensionTable} depends on the choice of the error bound $\epsilon$.
We repeated all calculations for various choices of $\epsilon$ and present the results reported in the last line of Tab.~\ref{tab:dimensionTable} for $\epsilon=10^{-5}$ and $\epsilon=10^{-7}$ in Tab.~\ref{tab:epsTable}.
As expected, the ratios in classes 1 and 5 increase for decreasing $\epsilon$, since we detect $\lb{\lambda}_\ANew < \lb{\lambda}_\Ge$ (as well as $\ub{\lambda}_\Ge < \ub{\lambda}_\ANew$, $\lb{\lambda}_\HR < \lb{\lambda}_\ANew$, and $\ub{\lambda}_\ANew < \ub{\lambda}_\HR$) for a larger number of examples (cf. \eqref{eq:ComparisonEps}).
However, beside this observation, the results are robust w.r.t.\ the value of $\epsilon$.  

\section{Conclusion}
\label{sec:Conclusion}

We significantly improved a method recently introduced in~\cite{Monnigmann2011a} for the efficient computation of spectral bounds
for Hessian matrices of twice continuously differentiable functions on hyperrectangles. 
The improvements build on the identification and utilization of sparsity that naturally arises in the first lines of every codelist for a function $\func:\R^n \rightarrow \R$.  

The improved method was applied to a set of 1522 examples previously analyzed in~\cite{SchulzeDarup2014JGO}.
The numerical results show that the consideration of sparsity results in significantly tighter eigenvalue bounds.
In fact, the improved method provided equally good or better eigenvalue bounds than Hertz and Rohn's method in $82.20\,\%$ of the examples while the corresponding percentage only reads $20.83\,\%$ for the original procedure.

In addition to illustrating the practical usefulness of the proposed improvements, we provided an important theoretic result.
In fact, it is well-known that the original method from~\cite{Monnigmann2011a} results in spectral bounds with $0 \in [\lambda_\func]$ for any function that involves the multiplications of two or more variables (see \cite[Rem. 4.3]{Monnigmann2011a}). Consequently, convex functions that involve such a multiplication cannot be detected to be convex with the original method. 
We showed that this restrictions does not apply for the improved method.

\section*{Acknowledgements}
Funding by Deutsche Forschungsgemeinschaft grant MO-1086/9 is gratefully acknowledged.

\bibliographystyle{abbrv}
\bibliography{literatur,mmo,msdLiterature} 

\appendix

\section{Supplementary results}

\begin{lemma}
\label{lem:specialMatrix}
Let $[a],[b],[c] \subset \R$ be real intervals, let $[\Lambda_\star([a], [b],[c])]$ be defined as in~\eqref{eq:intervalOperatorLambdaStar}, and
consider the matrix set 
$$\mathcal{H} = \{ H \in \R^{2 \times 2} \,|\, H_{11} \in [a], \,H_{22} \in [b],\, H_{12} \in [c],\, H=H^T \}.$$
Then, 
\begin{equation}
\label{eq:specialMatrixBounds}
\lb{\Lambda}_\star([a],[b],[c]) = \min_{H \in \mathcal{H}} \lambda_{\min}(H) \quad \textrm{and} \quad
\max_{H \in \mathcal{H}}  \lambda_{\max}(H) = \ub{\Lambda}_\star([a],[b],[c]).
\end{equation}
\end{lemma}
\begin{proof}
The eigenvalue bounds of a symmetric matrix $H=\left(\begin{smallmatrix}
a&c\\ c&b
\end{smallmatrix}\right) \in \mathcal{H}$  read 
$$
\lambda_{\min}(H) = \frac{1}{2}\, \left(  a +b - \sqrt{(a-b)^2 + 4\,c^2 } \right), \quad \lambda_{\max}(H) = \frac{1}{2}\, \left(  a +b + \sqrt{(a-b)^2 + 4\,c^2} \right).
$$
We therefore have to show that
\begin{align}
\label{eq:specialMatrixLowerBound}
 \lb{a} + \lb{b} - \sqrt{(\lb{a}-\lb{b})^2 + d} &= \min_{a \in [a], b\in[b], c\in[c]} a +b - \sqrt{(a-b)^2 + 4\,c^2 } \quad \text{and} \\
 \label{eq:specialMatrixUpperBound}
 \ub{a} + \ub{b} + \sqrt{(\ub{a}-\ub{b})^2 + d} &= \max_{a \in [a], b\in[b], c\in[c]} a +b + \sqrt{(a-b)^2 + 4\,c^2 },
\end{align}
where the l.h.s.~results from~\eqref{eq:intervalOperatorLambdaStar} and where $d = 4\,\max \{\lb{c}^2,\ub{c}^2\}$.
We show that~\eqref{eq:specialMatrixLowerBound} holds and claim~\eqref{eq:specialMatrixUpperBound} can be proven analogously.
First note that the r.h.s.~in~\eqref{eq:specialMatrixLowerBound} can be simplified to
\begin{equation}
\label{eq:eliminateC}
\min_{a \in [a], b\in[b], c\in[c]} a +b - \sqrt{(a-b)^2 + 4\,c^2 } = \min_{a \in [a], b\in[b]} a +b - \sqrt{(a-b)^2 + d }
\end{equation}
since $c$ only occurs in the radicand. Consider the function $f:\R^2 \rightarrow \R$, $f(a,b)= a+b - \sqrt{(a-b)^2 + d}$, which occurs on the r.h.s.~of~\eqref{eq:eliminateC}
and note that $f(a,b)$ is concave (since $g(a,b)=\sqrt{(a-b)^2 + d}$ is convex).
Since the hyperrectangle $\hyperRec = [a] \times [b]$ is convex, the minimum on the r.h.s.~of~\eqref{eq:eliminateC} is attained at one of the vertices of $\hyperRec$.
Among the candidate tuples $(\lb{a},\lb{b})$, $(\lb{a},\ub{b})$, $(\ub{a},\lb{b})$, and $(\ub{a},\ub{b})$, it is easy to show that $(\lb{a},\lb{b})$ results in the smallest function value, i.e., $f(\lb{a},\lb{b}) \leq \min \{f(\lb{a},\ub{b}),f(\ub{a},\lb{b}),f(\ub{a},\ub{b})\}$. Thus, \eqref{eq:specialMatrixLowerBound} holds.
\end{proof}

\begin{table}[htp] 
\vspace{-8mm}
\caption{Rules for the calculation of the eigenvalue bounds $[\lambda_k^\dagger]$ in the $k$-th line of the codelist \eqref{eq:codelist} according to Prop.~\ref{prop:newArithmetic}.
The intervals $[\lambda_t]$, $[\lambda_{i,0}]$ and $[\lambda_{j,0}]$ are shorthand notations for $[\lambda_t]=[\Lambda_t([\grad_{\LinVars_k^c} y_i],[\grad_{\LinVars_k^c} y_j])]$, $[\lambda_{i,0}]=[\min \{\lb{\lambda}_i^\dagger,0\}, \max \{\ub{\lambda}_i^\dagger,0\} ]$, and $[\lambda_{j,0}]=[\min \{\lb{\lambda}_j^\dagger,0\}, \max \{\ub{\lambda}_j^\dagger,0\} ]$.
The index sets $\LinVars_\cup$ and $\LinVars_\cap$ are shorthand notations for $\LinVars_\cup = \LinVars_i \cup \LinVars_j$ and $\LinVars_\cap = \LinVars_i \cap \LinVars_j$.  Condition $C_\star$ reads  $(\IndepOfVars_i \cup \IndepOfVars_j = \setN) \wedge (|\IndepOfVars_i|=n-1) \wedge (|\IndepOfVars_j|=n-1)$.}
\label{tab:newArithmetic}
    \centering
\small
\begin{tabular}{l@{\,\,}|l@{\,\,}|l@{\,\,}} 
  {\tt op} $\Phi_k$ &  $[\lambda_k^\dagger]$   & condition \\ 
  \hline \hline 
  {\tt add}    &     $[0,0]$   & $\LinVars_i = \setN  \wedge \LinVars_j =  \setN$ \\
 & $[\lambda_i^\dagger]$  & $\LinVars_i \subset \setN  \wedge \LinVars_j= \setN$ \\
 & $[\lambda_j^\dagger]$  & $\LinVars_i = \setN  \wedge \LinVars_j \subset  \setN$ \\
  & $[\Lambda_r([\lambda_i^\dagger],[\lambda_j^\dagger])]$   & $\LinVars_i \subset \setN  \wedge \LinVars_j \subset  \setN \wedge \LinVars_\cup = \setN$\\
 & $[\lambda_i^\dagger]+[\lambda_j^\dagger]$    & $\LinVars_\cup \subset \setN \wedge \LinVars_i = \LinVars_j$ \\ 
  & $[\lambda_i^\dagger]+[\lambda_{j,0}]$    & $\LinVars_\cup \subset \setN \wedge \LinVars_i \subset \LinVars_j$ \\ 
 & $[\lambda_{i,0}]+[\lambda_j^\dagger]$    & $\LinVars_\cup \subset \setN \wedge \LinVars_j \subset \LinVars_i$ \\ 
  & $[\lambda_{i,0}]+[\lambda_{j,0}]$    & $\LinVars_\cup \subset \setN \wedge \LinVars_i \nsubseteq \LinVars_j \wedge \LinVars_j \nsubseteq \LinVars_i$ \\ 
      \hline
  
  {\tt mul}          & $[\lambda_t]$ & $\LinVars_i = \setN  \wedge \LinVars_j= \setN$  \\
 &  $[\lambda_t]+[y_j]\,[\lambda_i^\dagger]$ & $\LinVars_i \subset \setN  \wedge \LinVars_j= \setN\wedge \LinVars_k = \LinVars_i $ \\ 
 &  $[\lambda_t]+[y_j]\,[\lambda_{i,0}]$ & $\LinVars_i \subset \setN  \wedge \LinVars_j= \setN \wedge \LinVars_k \subset \LinVars_i \wedge \neg C_\star$ \\
& $[\Lambda_\star([y_j]\,[\lambda_i^\dagger],[0,0],[\grad_{\IndepOfVars_i^c} y_i][\grad_{\IndepOfVars_j^c} y_j])]$ & $\LinVars_i \subset \setN  \wedge \LinVars_j= \setN  \wedge  C_\star $  \\ 
&  $[\lambda_t]+[y_i]\,[\lambda_j^\dagger]$ & $\LinVars_i = \setN  \wedge \LinVars_j \subset  \setN  \wedge \LinVars_k = \LinVars_j$  \\ 
 &  $[\lambda_t]+[y_i]\,[\lambda_{j,0}]$ & $\LinVars_i = \setN  \wedge \LinVars_j \subset  \setN \wedge \LinVars_k \subset \LinVars_j \wedge \neg C_\star$ \\
 & $[\Lambda_\star([0,0],[y_i]\,[\lambda_j^\dagger],[\grad_{\IndepOfVars_i^c} y_i][\grad_{\IndepOfVars_j^c} y_j])]$ & $\LinVars_i = \setN  \wedge \LinVars_j \subset \setN  \wedge  C_\star $  \\  
  &  $[\lambda_t]+[\Lambda_{r}([\Lambda_{r}([y_j]\,[\lambda_i^\dagger],[y_i]\,[\lambda_j^\dagger])],[0,0])]$ & $\LinVars_i \subset \setN  \wedge \LinVars_j \subset \setN \wedge\LinVars_\cup = \setN \wedge \LinVars_k \subset \LinVars_\cap $ \\
  &  $[\lambda_t]+[\Lambda_{r}([y_j]\,[\lambda_i^\dagger],[y_i]\,[\lambda_j^\dagger])]$  & $\LinVars_i \subset \setN  \wedge \LinVars_j \subset \setN \wedge \LinVars_\cup = \setN \wedge \LinVars_k = \LinVars_\cap \wedge \neg C_\star$ \\ 
 & $[\Lambda_\star([y_j]\,[\lambda_i^\dagger],[y_i]\,[\lambda_j^\dagger],[\grad_{\IndepOfVars_i^c} y_i][\grad_{\IndepOfVars_j^c} y_j])]$ & $\LinVars_i \subset \setN  \wedge \LinVars_j \subset \setN  \wedge  C_\star $  \\
  & $[\lambda_t]+[y_j]\,[\lambda_i^\dagger] +[y_i]\,[\lambda_j^\dagger]$ & $ \LinVars_\cup \subset \setN   \wedge \LinVars_k = \LinVars_i = \LinVars_j $  \\ 
 & $[\lambda_t]+[y_j]\,[\lambda_i^\dagger] +[y_i]\,[\lambda_{j,0}]$ & $ \LinVars_\cup \subset \setN   \wedge \LinVars_k = \LinVars_i \subset \LinVars_j $ \\ 
  & $[\lambda_t]+[y_j]\,[\lambda_{i,0}] +[y_i]\,[\lambda_j^\dagger]$ & $ \LinVars_\cup \subset \setN  \wedge \LinVars_k = \LinVars_j \subset \LinVars_i $ \\ 
   & $[\lambda_t]+[\Lambda_r([y_j]\,[\lambda_i^\dagger] +[y_i]\,[\lambda_j^\dagger],[0,0])]$ & $ \LinVars_\cup \subset \setN  \wedge \LinVars_k \subset \LinVars_i = \LinVars_j $ \\
 & $[\lambda_t]+[\Lambda_r([y_j]\,[\lambda_i^\dagger] +[y_i]\,[\lambda_{j,0}],[0,0])]$ & $ \LinVars_\cup \subset \setN   \wedge \LinVars_k \subset \LinVars_i \subset \LinVars_j $    \\ 
  & $[\lambda_t]+[\Lambda_r([y_j]\,[\lambda_{i,0}] +[y_i]\,[\lambda_j^\dagger],[0,0])]$ & $ \LinVars_\cup \subset \setN  \wedge \LinVars_k \subset \LinVars_j \subset \LinVars_i $ \\ 
 & $[\lambda_t]+[y_j]\,[\lambda_{i,0}] +[y_i]\,[\lambda_{j,0}]$ & $\LinVars_\cup \subset \setN \wedge  \LinVars_i \nsubseteq \LinVars_j \wedge \LinVars_j \nsubseteq \LinVars_i  $  \\  
\hline

  {\tt powNat}      & $m\,(m\!-\!1)\,[y_i]^{m-2}\,[\lambdaaaT([\grad_{\LinVars_k^c} y_i])]$  & $\LinVars_i = \setN$  \\ 
  &  $m\,[y_i]^{m-2}((m\!-\!1)[\lambdaaaT([\grad_{\LinVars_k^c} y_i])]\!+\![y_i]\,[\lambda_i^\dagger])$  & $\LinVars_i \subset \setN \wedge \LinVars_k = \LinVars_i$ \\ 
  &  $m\,[y_i]^{m-2}((m\!-\!1)[\lambdaaaT([\grad_{\LinVars_k^c} y_i])]\!+\![y_i]\,[\lambda_{i,0}])$  & $\LinVars_i \subset \setN \wedge \LinVars_k \subset \LinVars_i$ \\ 
\hline
  {\tt oneOver}      &  $2\,[y_k]^3\,[\lambdaaaT([\grad_{\LinVars_k^c} y_i])]$  & $\LinVars_i = \setN$ \\  
 &  $[y_k]^2\,(2\,[y_k]\,[\lambdaaaT([\grad_{\LinVars_k^c} y_i])]-[\lambda_i^\dagger])$&  $\LinVars_i \subset \setN \wedge \LinVars_k = \LinVars_i$ \\ 
 &  $[y_k]^2\,(2\,[y_k]\,[\lambdaaaT([\grad_{\LinVars_k^c} y_i])]-[\lambda_{i,0}])$&  $\LinVars_i \subset \setN \wedge \LinVars_k \subset \LinVars_i$  \\ 
\hline
  {\tt sqrt}       & $1/(-4\,[y_k]^3)\,[\lambdaaaT([\grad_{\LinVars_k^c} y_i])]$   & $\LinVars_i = \setN$ \\ 
  & $1/(2\,[y_k])(1/(-2\,[y_i])[\lambdaaaT([\grad_{\LinVars_k^c} y_i])]\!+\![\lambda_i^\dagger])$ & $\LinVars_i \subset \setN \wedge \LinVars_k = \LinVars_i$\\   
  & $1/(2\,[y_k])(1/(-2\,[y_i])[\lambdaaaT([\grad_{\LinVars_k^c} y_i])]\!+\![\lambda_{i,0}])$ & $\LinVars_i \subset \setN \wedge \LinVars_k \subset \LinVars_i$ \\  
\hline
  {\tt exp}       & $[y_k]\,[\lambdaaaT([\grad_{\LinVars_k^c} y_i])]$  & $\LinVars_i = \setN$ \\ 
 & $[y_k]\,([\lambdaaaT([\grad_{\LinVars_k^c} y_i])]+[\lambda_i^\dagger])$  & $\LinVars_i \subset \setN \wedge \LinVars_k = \LinVars_i$ \\
 & $[y_k]\,([\lambdaaaT([\grad_{\LinVars_k^c} y_i])]+[\lambda_{i,0}])$  & $\LinVars_i \subset \setN \wedge \LinVars_k \subset \LinVars_i$    \\
\hline
  {\tt ln}    &     $-1/[y_i]^2\,[\lambdaaaT([\grad_{\LinVars_k^c} y_i])]$  &$\LinVars_i = \setN$ \\   
&  $1/[y_i]\,([\lambda_i^\dagger]-1/[y_i]\,[\lambdaaaT([\grad_{\LinVars_k^c} y_i])])$  & $\LinVars_i \subset \setN \wedge \LinVars_k = \LinVars_i$\\
 &  $1/[y_i]\,([\lambda_{i,0}]-1/[y_i]\,[\lambdaaaT([\grad_{\LinVars_k^c} y_i])])$  & $\LinVars_i \subset \setN \wedge \LinVars_k \subset \LinVars_i$ \\
    \hline 
  {\tt addC}     &  $[0,0]$  & $\LinVars_i = \setN$   \\
 &  $[\lambda_i^\dagger]$  & $\LinVars_i \subset \setN$   \\
\hline
  {\tt mulByC}   &  $[0,0]$  & $\LinVars_i = \setN$   \\
 & $c\,[\lambda_i^\dagger]$ &  $\LinVars_i \subset \setN$  \\
\end{tabular}  
\end{table} 

\end{document}